\documentclass[a4paper,10pt]{amsart}
\usepackage{amsmath,amsthm,amssymb}
\usepackage[latin1]{inputenc}
\usepackage[T1]{fontenc}
\usepackage[all]{xy}
\usepackage{mathrsfs}
\usepackage{hyperref}

\numberwithin{equation}{section}
\newtheorem{thm}[equation]{Theorem}
\newtheorem*{thm*}{Theorem}
\newtheorem{cor}[equation]{Corollary}
\newtheorem{lem}[equation]{Lemma}
\newtheorem{prop}[equation]{Proposition}
\newtheorem*{schol*}{\textbf{\emph{Sch\'olion}}}
\newtheorem{conj}[equation]{Conjecture}
\theoremstyle{definition}
\newtheorem{ex}[equation]{Example}
\newtheorem{rem}[equation]{Remark}
\newtheorem{defn}[equation]{Definition}

\DeclareMathOperator{\h}{H} 
\DeclareMathOperator{\sgn}{sgn} 
 
\newcommand{\triv}{{\mathbf{1}}}
\def\R{\mathbb R}

\def\Z{\mathbb Z}
\def\A{\mathbb A}
\def\Q{\mathbb Q}
\def\C{\mathbb C}
\def\F{\mathbb F}
\def\E{\mathbb E}
\def\K{\mathbb K}

\def\ira{\stackrel{\sim}{\rightarrow}}
\def\hra{\hookrightarrow}

\def\ra{\rightarrow}

\def\H{\mathbb H}

\def\g{\mathfrak g}
\def\gl{\mathfrak{gl}}
\def\t{\mathfrak t}

\def\a{\mathfrak a}

\def\q{\mathfrak q}

\def\h{\mathfrak h}
\def\l{\mathfrak l}
\def\m{\mathfrak m}
\def\k{\mathfrak k}
\def\p{\mathfrak p}
\def\u{\mathfrak u}
\def\<{\langle}
\def\>{\rangle}

\title[]%
{On some arithmetic properties of automorphic forms of \\
$GL_m$ over a division algebra}
\author{Harald Grobner \ \ \and \ \ A. Raghuram}
\address{Harald Grobner: Fakult\"at f\"ur Mathematik\\ Universit\"at Wien\\ Oskar--Morgenstern--Platz 1\\ A-1090 Wien\\Austria}
\email{harald.grobner@univie.ac.at}

\address{A. Raghuram: Indian Institute of Science Education and Research\\ Pashan, Pune 411021\\ India}
\email{raghuram@iiserpune.ac.in}

\thanks{H.G. is supported by the Austrian Science Fund (FWF) Erwin Schr\"odinger grant, J 3076-N13. A.R. is partially supported by the National Science Foundation (NSF), award number DMS-0856113, and an Alexander von Humboldt Research Fellowship.}
\thanks{This is a version of a paper to appear in Int. J. Number Th., DOI:10.1142/S1793042114500110}

\begin{document}

\begin{abstract}
In this paper we investigate arithmetic properties of automorphic forms on the group $G' = GL_m/D$, for a central division-algebra $D$ over an arbitrary number field $F$. The results of this article are generalizations of results in the split case, i.e., $D=F$, by Shimura, Harder, Waldspurger and Clozel for square-integrable automorphic forms and also by Franke and Franke-Schwermer for general automorphic representations. We also compare our theorems on automorphic forms of the group $G'$ to statements on automorphic forms of its split form using the global Jacquet-Langlands correspondence developed by Badulescu and Badulescu-Renard. Beside that we prove that the local version of the Jacquet-Langlands transfer at an archimedean place preserves the property of being cohomological.
\end{abstract}
\maketitle

\setcounter{tocdepth}{1}
\tableofcontents

\section{Introduction and statements of results}
Let $D$ be a central division algebra of index $d$ over a number field $F$. The group $G'=GL_m/D$ of invertible $m\times m$ - matrices with entries in $D$ defines a connected, reductive group over $F$ and is an inner form of the split general linear group $G=GL_{n}/F$, $n=dm$. In the split case, i.e.,
if $D=F$, many important results on the arithmetic theory of automorphic forms are known due to several people. Within the scope of the present paper, for cusp forms one should particularly mention the work of Shimura \cite{shimura} for $n=2$ and $F$ totally real, Harder \cite{hardergl2} and Waldspurger \cite{waldsp} for $n=2$ and any $F$, Clozel \cite{clozel} for general $n$ and $F$; and for general automorphic forms of the group $G$ we mention Franke \cite{franke} and Franke--Schwermer \cite{schwfr}. In \cite{waldsp} one may also find results in the very special non-split case $m=1$ and $d=2$.

The main aim of this article is to study the arithmetic of automorphic forms on $G'$ and to generalize some of the results of the above mentioned people to
the case of $GL_m/D$ for a general $m$ and a general $D$.

For a central division algebra $D$ over $F$ and $m\geq 1$, Badulescu \cite{ioan} and Badulescu--Renard \cite{ioanrenard} have recently proved the existence of the global Jacquet-Langlands transfer $JL$ from discrete series automorphic representations of $G'(\A)$ to discrete series automorphic representations of $G(\A)$. This establishes a special instance of Langlands functoriality and forms an important instrument for the analysis of arithmetic properties of square-integrable automorphic forms of $G'(\A)$.
We also explicitly describe the interplay of the local Jacquet-Langlands transfer and cohomology.

Our first theorem, which is of an arithmetic nature, deals with the notion of a regular algebraic representation of $G'(\A)$. In the split case,  i.e., $D=F$, Clozel gave a definition of a representation being regular algebraic: A discrete series automorphic representation $\Pi$ of $G(\A)$ is called algebraic, if each of the irreducible representations in the archimedean part $\Pi_\infty$ corresponds via the local Langlands parametrization to a sum of $n$ algebraic characters of $\C^*.$ An algebraic representation $\Pi$ is furthermore called regular, if the infinitesimal character of $\Pi_\infty$ is regular. It easily follows from Clozel \cite{clozel} that a cuspidal automorphic representation $\Pi$ is regular algebraic if and only if $\Pi_\infty$ is essentially tempered and cohomological with respect to a certain algebraic coefficient system $E_\mu$.

In this paper we extend Clozel's notion as follows: We call a discrete series automorphic representation $\Pi'=\Pi'_\infty\otimes\Pi'_f$ of $G'(\A)$ algebraic (resp., regular algebraic) if its global Jacquet--Langlands transfer $\Pi=JL(\Pi')$ is. With this definition we prove the following generalization of Clozel's result, cf. Thm.\ \ref{thm:regalgG'}.

\begin{thm}
Let $\Pi'$ be a discrete series automorphic  representation of $G'(\A)$ and assume that $JL(\Pi')$ is cuspidal. Then the following are equivalent:
\begin{enumerate}
\item[(i)] $\Pi'$ is regular algebraic.
\item[(ii)] $\Pi'_\infty$ is cohomological and essentially tempered.
\end{enumerate}
\end{thm}

Furthermore, the implication (ii)$\Rightarrow$(i) does not seem to need the assumption that $JL(\Pi')$ is cuspidal. See Remark \ref{ex:(ii)->(i)}. In contrast, the implication (i)$\Rightarrow$(ii) may fail without assuming $JL(\Pi')$ is cuspidal, cf. Ex. \ref{ex:(i)not(ii)}.

Given our definition of an algebraic representation, we also generalize Clozel's ``Lemme de puret\'e'' to the case of $G'$, cf. Lem.\ \ref{lem:pureteG'}:

\begin{lem}[Purity Lemma for $G'$]
If $\Pi'$ is an algebraic representation of $G'(\A)$ and $JL(\Pi')$ is cuspidal, then there is a {\sf w} $\in\Z$ such that for all archimedean places $v$, the algebraic characters of $\C^*$ associated to $JL(\Pi')_v|\cdot|_v^{\frac{1-n}{2}}$ are of the form $z\mapsto z^p(\overline z)^q$ with $p+q=$ {\sf w}.
\end{lem}

Again, we give an example that the cuspidality assumption on $JL(\Pi')$ cannot be removed from the statement of the lemma, see Ex. \ref{ex:purity}.

Next, we analyze the interplay of cohomological automorphic forms of $G'(\A)$ and the action of Aut$(\C)$. Following Waldspurger \cite{waldsp} and Clozel \cite{clozel}, for any representation $\nu$ of $G'(\A_f)$ and $\sigma\in$ Aut$(\C)$, there is the $\sigma$-twisted representation ${}^\sigma\!\nu:=\nu\otimes{}_{\sigma}\C$. In particular, this definition applies to the finite part $\Pi'_f$ of an automorphic representation $\Pi$ of $G'(\A)$. The action of $\sigma\in\textrm{Aut}(\C)$ on a finite-dimensional highest weight module $E_\mu$ of $G'_\infty$ is via its permutation action on the embeddings of $F$ into $\C$. Now, let
$$
S_{G'}= G'(F)\backslash G'(\A)/K'^\circ_\infty,
$$
where $K'^\circ_{\infty}$ is the topological connected component of the product of the center of $G'_\infty$ and a maximal compact subgroup of $G'_\infty$. A finite-dimensional highest weight module $E_\mu$ defines a sheaf $\mathcal E_\mu$ on $S_{G'}$. The corresponding sheaf cohomology $H^q(S_{G'},\mathcal E_\mu)$ can be computed using Betti-cohomology and hence for any $\sigma\in$ Aut$(\C)$, there is a $\sigma$-linear isomorphism
$$
\sigma^*: H^q(S_{G'},\mathcal E_\mu)\ira H^q(S_{G'},{}^\sigma\!\mathcal E_\mu).
$$
On the other hand, $H^q(S_{G'},\mathcal E_\mu)$ is isomorphic to the $(\g'_\infty,K'^\circ_\infty)$-cohomology of the space of automorphic forms on $G'(\A)$, which follows from Franke \cite{franke}, and inherits from that a decomposition of $G'(\A_f)$-modules
$$
H^q(S_{G'},\mathcal E_\mu) \cong\bigoplus_{\{P'\}}\bigoplus_{\varphi_{P'}} H_{\{P'\},\varphi_{P'}}^q(G',E_\mu),
$$
cf. \cite{schwfr} and \cite{moewal}. Here, the sums are ranging over the associate classes of parabolic $F$-subgroups $\{P'\}$ of $G'$ and attached associate classes of cuspidal automorphic representations $\varphi_{P'}$ of the corresponding Levi subgroup $L'$. In particular, the summand indexed by $\{G'\}$ gives the cohomology of the space of cuspidal automorphic forms of $G'(\A)$, which is usually called the cuspidal cohomology of $G'$.

We prove the following result, cf. Thm.\ \ref{thm:3.13'frankeschw}, which says that for regular highest weights, $\sigma^*$ respects this fine decomposition into cuspidal supports.

\begin{thm}
Let $E_\mu$ be a regular highest weight representation of $G'_\infty$ and $\sigma\in \textrm{\emph{Aut}}(\C)$. For each associate class of parabolic $F$-subgroups $\{P'\}$, and each associate class of cuspidal automorphic representations $\varphi_{P'}$, the summand $H^q_{\{P'\},\varphi_{P'}}(G',E_\mu)$ of
the global cohomology group $H^q(S_{G'},\mathcal E_\mu)$, is mapped by $\sigma^*$ isomorphically onto the summand $H^q_{\{P'\},{}^\sigma\!\varphi_{P'}}(G',{}^\sigma\!E_\mu)$ of $H^q(S_{G'},{}^\sigma\!\mathcal E_\mu)$ for a unique associate class ${}^\sigma\!\varphi_{P'}$:
\begin{displaymath}
\xymatrix{
H^q_{\{P'\},\varphi_{P'}}(G',E_\mu)\ar[rr]^\cong_{\sigma*} & & H^q_{\{P'\},{}^\sigma\!\varphi_{P'}}(G',{}^\sigma\!E_\mu).                   }
\end{displaymath}
If $H^q_{\{P'\},\varphi_{P'}}(G',E_\mu)\neq 0$ and a representative $P'$ in the associate class $\{P'\}$ has Levi factor $L'$ and $\varphi_{P'}$ is represented by a cuspidal automorphic representation $\Pi'$ of $L'(\A)$, then $\Pi'\otimes\rho_{P'}$ is cohomological. The $\sigma$-twist ${}^\sigma\Pi'_f\otimes {}^\sigma\!\rho_{P'_f}$ of its finite part is the finite part of a unique cohomological cuspidal automorphic representation $\Xi'$ and the associate class ${}^\sigma\!\varphi_{P'}$ is uniquely determined by the representation $\Xi'\otimes\rho_{P'}^{-1}$.
\end{thm}

Under the assumption of the regularity of $E_\mu$, this theorem is a generalization of the analogous result of Franke--Schwermer \cite{schwfr} in the split case, $D=F$. In this setup, it is also a simultaneous generalization of Franke's \cite[Theorem 20]{franke} on the compatibility of $\sigma^*$ with the $\{P\}$-decomposition and Clozel's \cite[Th\'eor\`em 3.13]{clozel}, which states that the $\sigma$-twist ${}^\sigma\Pi_f$ of the finite part of a cohomological cuspidal automorphic representations $\Pi$ of $G(\A)$ is the finite part of a cohomological, cuspidal automorphic representation.

If we consider the action of ${\rm Aut}(\C)$ on a regular algebraic representation and ask whether it is compatible with
the global Jacquet-Langlands transfer $JL$, then we obtain the following result, see Thm.\ \ref{thm:3.13'}:

\begin{thm}
Let $\Pi'$ be a regular algebraic, cuspidal automorphic representation of $G'(\A)$ and assume that $JL(\Pi')$ is cuspidal. For all $\sigma\in \textrm{\emph{Aut}}(\C)$, there is a unique $\sigma$-twisted representation ${}^\sigma\Pi'$ of $G'(\A)$, which is regular algebraic and such that $JL({}^\sigma\Pi')$ is cuspidal. The action of $\textrm{\emph{Aut}}(\C)$ commutes with taking the global Jacquet-Langlands transfer, i.e., ${}^\sigma\!JL(\Pi')=JL({}^\sigma\Pi')$ for all $\sigma\in \textrm{\emph{Aut}}(\C)$.
\end{thm}

In the final section we also prove some arithmetic results on the rationality field
$$\Q(\Pi'_f)=\{z\in\C \ | \ \sigma(z)=z\quad\forall\sigma\in\textrm{Aut}(\C)\textrm{ for which } {}^\sigma\Pi'_f\cong\Pi'_f\}$$
of the finite part of a cuspidal automorphic representation $\Pi'$ of $G'(\A)$ which is cohomological. At infinity, define $\Q(\mu)$ to be a minimal extension of the fixed-field in $\C$ of those $\sigma\in$ Aut$(\C)$ which fix a highest weight representation $E_\mu$ with respect to which $\Pi'_\infty$ has non-zero $(\g'_\infty,K'^\circ_\infty)$-cohomology, minimal such that $D$ splits over $\Q(\mu)$. Now, let the field $\Q(\Pi')$ be the compositum of $\Q(\mu)$ and $\Q(\Pi'_f)$. This is a generalization of the analogous notation used in Raghuram--Shahidi, cf.\ \cite{raghuram-shahidi-imrn}. The following theorem, contained in Thm.\ \ref{thm:Q(pi_f)numberfield}, Thm.\ \ref{prop:cuspratio} and Prop.\ \ref{prop:Q=QJL}, generalizes the analogous results in the split case as well as Waldspurger's corresponding theorems for $d=2$ and $m=1$.

\begin{thm}
Let $\Pi'$ be a cuspidal and cohomological representation of $G'(\A)$. Then $\Q(\Pi')$ is a number field and $\Pi'_f$ admits a $G'(\A_f)$-invariant $\Q(\Pi')$-structure. In particular, $\Pi'_f$ is defined over a number field. Furthermore, if $\Pi'$ is regular algebraic and $JL(\Pi')$ is cuspidal, then there is the equality of fields
$$\Q(\Pi'_f)=\Q(JL(\Pi')_f).$$
\end{thm}

%
%
In view of the results of this paper, we may generalize Clozel's \cite[Conjectures 3.7 and 3.8]{clozel} as:

\begin{conj}
Let $\Pi'$ be a cuspidal automorphic representation of $G'(\A)$. Then the following are equivalent:
\begin{enumerate}
\item[(i)] $\Pi'_f$ is defined over a number field.
\item[(ii)] $\Pi'$ is algebraic.
\end{enumerate}
\end{conj}

Beside the above theorems on the arithmetic properties of automorphic forms of $G'(\A),$
we also prove in Sections \ref{sect:irrunitcoh} and \ref{sect:local-global-jl} a number of purely representation-theoretical results on
$GL_k(\H)$ and $GL_n(\R)$, $n=2k$. To begin, we describe an explicit classification of the cohomological irreducible, unitary dual of
$GL_k(\H)$ and $GL_{n}(\R)$, $n=2k\geq 2$, following Vogan--Zuckerman \cite{vozu}. This is contained in Thm.\ \ref{prop:irrunitG'} and Thm.\ \ref{prop:irrunitG} which describe the set of all cohomological, irreducible, unitary $A_{\q}(\lambda)$-modules of $GL_k(\H)$ and $GL_{n}(\R)$ with respect to any finite-dimensional coefficient system $E_\mu$ very concretely in terms of certain ordered partitions of $k$ and $n$. The weight $\lambda$ depends on the weight $\mu$ as in Definition \ref{def:lambda-mu} below. (This classification is also in accordance with the results of Speh, cf. \cite{speh}.) Having parameterized the cohomological irreducible, unitary dual Coh$_\mu(GL_k(\H))$ of $GL_k(\H)$ and Coh$_\mu(GL_n(\R))$ of $GL_{n}(\R)$ with respect to $E_\mu$ by such partitions, we then prove the following theorem, cf. Thm.\ \ref{thm:irrunitJL}.

\begin{thm}
At a non-split archimedean place $v$ of $F$, the local Jacquet-Langlands transfer $|LJ|_v$, constructed by Badulescu--Renard, defines a surjective map
$$\xymatrix{\textrm{\emph{Coh}}_{\mu}(GL_n(\R))\ar@{->>}[rr]^{|LJ|_v}  & & \textrm{\emph{Coh}}_{\mu}(GL_k(\H))}$$
given explicitly by
$$|LJ|_v(A_{\q_{\underline n}}(\lambda)\otimes\sgn^\varepsilon)=A_{\q'_{\underline k}}(\lambda).$$
Here $\underline n$ (resp., $\underline k$) stands for the partition $n=\sum_{i=0}^rn_i$ (resp., $k=\sum_{i=0}^rk_i$) with $n_i=2k_i$, $0\leq i\leq r$ and $k_i>0$ for $1\leq i\leq r$.

Moreover, $|LJ|_v$ maps tempered cohomological representations to tempered cohomological representations.
\end{thm}

Using the parametrization by ordered partitions, it is an easy exercise to determine the fibers of $|LJ|_v$ in Coh$_\mu(GL_n(\R))$ over a given $A_{\q'}(\lambda)$-module.
This theorem fits very well with the interplay of cohomological automorphic representations with Langlands functoriality as discussed in
Raghuram--Shahidi  \cite[Section 5.2]{ragsha}.

\medskip

{\small
\noindent{\it Acknowledgements:}
A.R. thanks David Vogan for some email correspondence in 2003 concerning representations of ${\rm GL}_m(D)$ with cohomology which was the genesis of this project. However, at that time the generalized Jacquet-Langlands correspondence was still not proved and one had to wait for Badulescu's theorems.
This project really got started when both the authors met at the Erwin Schr\"odinger Institute (ESI) in Vienna in February 2009. Both H.G. and A.R. thank the
ESI, and also the Max-Planck Institut f\"ur Mathematik for their hospitality. H.G. also thanks the Department of Mathematics of the Oklahoma State University, where much of this work was done, and the Institut de Math\'ematiques de Jussieu.}

\section{The general linear group and its inner forms}
\subsection{Generalities on division algebras}\label{sect:gendivalg}
Let $F$ be a number field whose set of all places is denoted $V=V_\infty\cup V_f$, where as usual $V_\infty$ is the subset of archimedean places and $V_f$ the subset of non-archimedean places. The local completion of $F$ at a place $v\in V$ is written $F_v$.

Let $D$ be a central division-algebra over $F$ of index $d$, i.e., $d^2=\dim_F D$. The local algebras $D_v=D\otimes_F F_v$ are central simple algebras over $F_v$ and hence isomorphic to a matrix algebra $M_{r_v}(A_v)$, for some integer $r_v\geq 1$ and a central division algebra $A_v$ over $F_v$. The algebra $D$ is said to be split at $v$ if $A_v=F_v$ and non-split at $v$ otherwise, i.e., $A_v$ is not a field. The set of non-split places is finite. Analogous to the global situation, let $d_v$ be the index of $D_v$, i.e., $d_v^2=\dim_{F_v}A_v$. Then $r_v d_v=d$ for all $v$. If $v\in V_\infty$ is real then $d_v\in\{1,2\}$, i.e., $A_v=\R$ or $\H$ and $D_v=M_d(\R)$ if $v$ is split and $M_{d/2}(\H)$ is $v$ is non-split (in which case $d$ is even). Given any $m\geq 1$ we set $n:=dm$ and $k:=n/2$.

\subsection{The groups $G'$ and $G$}\label{sect:groups}
The determinant
$\det'$ of an $m\times m$-matrix $X\in M_m(D)$, $m\geq 1$, is
the generalization of the reduced norm to matrices:
$\det'(X):=\det(\varphi(X\otimes 1))$, for some isomorphism
$\varphi: M_m(D)\otimes_F\overline\Q\ira M_{n}(\overline\Q)$. It
is independent of $\varphi$ and is an $F$-rational polynomial in the coordinates of the entries of $X$. So the
group
$$G'(F)=\{X\in M_m(D)|\textrm{det}'(X)\neq 0\}$$
defines an algebraic group $GL_m'$ over $F$. It is
reductive and is an inner $F$-form of the split group $G:=GL_n/F$.
At a real place $v\in V_\infty$ we hence obtain $G'(\R)=GL_n(\R)$ if $v$ is split and $G'(\R)=GL_k(\H)$ if $v$ is not split. We use the notation $G'_v:=G'(F_v)$ and $G_v:=G(F_v)$ for $v\in V$ and set as usual $G'_\infty:=\prod_{v\in V_\infty} G'_v$, resp., $G_\infty:=\prod_{v\in V_\infty} G_v$. Lie algebras of Lie groups are denoted by the same but gothic letter, e.g. $\g_v=Lie(G_v)$, $\g_\infty=Lie(G_\infty)$.

\subsection{Finite-dimensional representations}\label{sect:highweights}
We fix once and for all a maximal $F$-split torus $T$ in $G$. The group $T_\infty=\prod_{v\in V_\infty} T_v=\prod_{v\in V_\infty} T(F_v)$ is then a Cartan subgroup of $G_\infty$. Fixing the set of dominant algebraic characters $X^+(T_\infty)$ of $T_\infty$ in the usual way, gives us that a tuple $\mu=(\mu_v)_{v\in V_\infty}\in X^+(T_\infty)$ can be identified with an equivalence class of irreducible, algebraic, finite-dimensional representations $E_\mu$ of $G_\infty$ (on complex vector spaces) via the highest weight correspondence. It is clear that any such representation $E_\mu$ factors into irreducible representations $E_\mu=\bigotimes_{v\in V_\infty} E_{\mu_v}$, where $E_{\mu_v}$ is the irreducible representation of $G_v$ of highest weight $\mu_v$. If $v$ is real, then $\mu_v=(\mu_{v,1},...,\mu_{v,n})$ with $\mu_{v,1}\geq ...\geq\mu_{v,n}$ and $\mu_{v,i}\in\Z$, for $1\leq i\leq n$; and if $v$ is complex, corresponding to the complex embeddings $\{\iota_v,\bar\iota_v\}$ of $F$, then $\mu_v$ is given by a pair $(\mu_{\iota_v},\mu_{\bar\iota_v})$ of $n$-tuples of the above form. A representation $E_\mu$ is called {\it essentially self-dual} if all its local factors $E_{\mu_{v}}$ are, i.e., if for all $v\in V_\infty$ there is a $w_v\in\Z$ such that $E_{\mu_v}\cong E^{\sf v}_{\mu_v}\otimes\det^{w_v}$. At a real place this reads as
$$\mu_{v,i}+\mu_{v,n-i+1}=w_v, \quad\quad 1\leq i\leq n$$
and at a complex place this means
$$\mu_{\bar\iota_v,i}+\mu_{\iota_v,n-i+1}=w_v, \quad\quad 1\leq i\leq n.$$
It is called {\it self-dual} if $w_v=0$, i.e., $E_{\mu_v}\cong E^{\sf v}_{\mu_v}$.

As $G'_\infty$ is a real inner form of $G_\infty$ the notion of highest weights and irreducible finite-dimensional representations is defined via the passage to the split form $G_\infty$. That means that we say at a non-split place $v\in V_\infty$ an irreducible finite-dimensional representation $E'_v$ of $G'_v$ is {\it of highest weight} $\mu'_v$, if the complexified representation $E_v$ of $\g'_v\otimes\C=\mathfrak{gl}_n(\C)$ is. We hence drop the prime for such representations and write simply $E'_v=E_{\mu_v}$. Then, everything said above on representations of $G_\infty$ also applies to irreducible finite-dimensional complex representations $E_\mu=\bigotimes_{v\in V_\infty} E_{\mu_v}$ of $G'_\infty$ without changes (only adopting the notation of $\det$ to $\det'$). A highest weight representation $E_\mu$ is called {\it regular}, if $\mu$ lies in the interior of the dominant Weyl chamber of $G'_\infty$. The smallest algebraically integral element in the interior of the dominant Weyl chamber of $G'_\infty$ and $G_\infty$ is given by $\rho=(\rho_v)_{v\in V_\infty}$ with $\rho_v=(\frac{n-1}{2},\frac{n-3}{2},...,-\frac{n-1}{2})$ for all real places $v\in V_\infty$ (resp. the pair $\rho_v=(\rho_{\iota_v},\rho_{\bar\iota_v})$, $\rho_{\iota_v}=\rho_{\bar\iota_v}=(\frac{n-1}{2},\frac{n-3}{2},...,-\frac{n-1}{2})$, if $v$ is complex).

\subsection{$(\g_\infty',K'^\circ_\infty)$-cohomology}\label{sect:gKcoh}
Let $Z'/F$ be the center of the algebraic group $G'/F$ and denote
$Z'_\infty=\prod_{v\in V_\infty} Z'_v=\prod_{v\in V_\infty} Z'(F_v)$.
At an archimedean place $v\in V_\infty$ we let $K'_v$ be the product of a maximal compact subgroup of the real Lie group $G'_v$ and $Z'_v$. Explicitly, we get
$$K'_v=\left\{
\begin{array}{ll}
Sp(k)\R^* & \textrm{if $v$ non-split}\\
O(n)\R^* & \textrm{if $v$ split and real }\\
U(n)\C^* & \textrm{if $v$ complex,}
\end{array}
\right.$$
and define $K'_\infty=\prod_{v\in V_\infty} K'_v$. Analogously, we set $K_\infty=\prod_{v\in V_\infty} K_v$ where $K_v:=K'_v$ at split places and $K_v:=O(n)\R^*$ at non-split places. By $K'^\circ_\infty$ (resp., $K^\circ_\infty$) we denote the topological connected component of the identity within $K'_\infty$ (resp., $K_\infty$). Hence, locally
$$K'^\circ_v=\left\{
\begin{array}{ll}
Sp(k)\R_+ & \textrm{if $v$ non-split}\\
SO(n)\R_+ & \textrm{if $v$ split and real }\\
U(n)\C^* & \textrm{if $v$ complex.}
\end{array}
\right.$$
We assume familiarity with the basic facts and notions concerning $(\g_\infty',K'^\circ_\infty)$-modules (and $(\g_\infty,K^\circ_\infty)$-modules), to be found in the book of Borel--Wallach \cite{bowa}, 0, I. All Lie-group representations $\Pi'_\infty=\bigotimes_{v\in V_\infty} \Pi'_v$ of $G'_\infty$ appearing in this paper define a $(\g_\infty',K'^\circ_\infty)$-module and hence $(\g_v',K'^\circ_v)$-modules, which we shall all denote by the same letter as the original Lie group representation. In particular, this applies to a highest weight representation  $E_\mu=\bigotimes_{v\in V_\infty} E_{\mu_v}$. If furthermore, $\Pi'_\infty=\bigotimes_{v\in V_\infty} \Pi'_v$ is any $(\g_\infty',K'^\circ_\infty)$-module, then we denote by
$$H^q(\g_\infty',K'^\circ_\infty,\Pi'_\infty)$$
its space of $(\g_\infty',K'^\circ_\infty)$-cohomology (in degree $q$), cf. \cite{bowa}, I.5. A module $\Pi'_\infty$ is called {\it cohomological}, if there is a highest weight representation $E_\mu$ as in Section \ref{sect:highweights}, such that $H^q(\g_\infty',K'^\circ_\infty,\Pi'_\infty\otimes E_\mu)\neq 0$ for some degree $q$. It is a basic fact that these cohomology groups obey the K\"unneth-rule, i.e.,
$$H^q(\g_\infty',K'^\circ_\infty,\Pi'_\infty\otimes E_\mu)\cong\bigoplus_{\sum_v q_v=q}\quad\bigotimes_{v\in V_\infty} H^{q_v}(\g'_v,K'^\circ_v,\Pi'_v\otimes E_{\mu_v}).$$
Hence, $\Pi'_\infty$ is cohomological, if and only if all its local components $\Pi'_v$ are, i.e., they have non-vanishing $(\g_v',K'^\circ_v)$-cohomology with respect to some local highest weight representation $E_{\mu_v}$.

\section{Generalities on automorphic representations of $G'(\A)$}\label{sect:automf}
\subsection{}
We call an irreducible sub-quotient $\Pi'$ of the space $\mathcal A(G'(F)\backslash G'(\A))$ of automorphic forms an {\it automorphic representation} of $G'(\A)$ (although $\Pi'$ is strictly speaking not a $G'(\A)$-module), cf. the article of Borel--Jacquet \cite{bojac} 3-4. Let $\R_+$ be the multiplicative group of positive real numbers, viewed as a subgroup
$$\R_+\hookrightarrow G'(\A)$$
by embedding it diagonally into $G'_\infty$. Throughout this paper, we will identify the quotient $\R_+\backslash G'(\A)$, with $G'(\A)^{(1)}=\ker H_G$, $H_G$ being the Harish-Chandra height function $G'(\A)\ra\C$, cf.\ \cite{franke}, p. 185. Doing so, $\R_+ G'(F)\backslash G'(\A)$ has finite volume and it therefore makes sense to talk about subspaces of square-integrable automorphic forms in $\mathcal A(\R_+G'(F)\backslash G'(\A))$. Now, recall that by its very definition, every automorphic form is annihilated by some power of an ideal $\mathcal J$ of finite codimension in the center of the universal enveloping algebra of $\g'_\C=\g'_\infty\otimes_{\R}\C$. Let us fix such an ideal $\mathcal J$ and denote by
$$\mathcal A_\mathcal J(G')\subset\mathcal A(\R_+G'(F)\backslash G'(\A))$$
the $G'(\A)$-submodule consisting of those automorphic forms which are annihilated by some power of $\mathcal J$.

For later use, we will now recall a fine decomposition of the latter space $\mathcal A_\mathcal J(G')$, which was developed by Franke--Schwermer in \cite{schwfr} Thm.\ 1.4 and also in a similar way by M\oe glin--Waldspurger in \cite{moewal} III, Thm.\ 2.6., taking into account the so-called {\it parabolic support} and the {\it cuspidal support} of an automorphic representation. As a first step, $\mathcal A_\mathcal J(G')$ can be decomposed as a $G'(\A)$-module into a finite direct sum, cf. \cite{schwfr} 1.1.(4),
\begin{equation}\label{eq:autdecP'}
\mathcal A_\mathcal J(G')\cong\bigoplus_{\{P'\}} \mathcal A_{\mathcal J,\{P'\}}(G'),
\end{equation}
ranging over the set of all {\it associate classes} $\{P'\}$ of a parabolic $F$-subgroup $P'$ of $G'$. (Recall therefore that two parabolic $F$-subgroups $P_1'$ and $P_2'$ of $G'$ are called associate, if their Levi-factors $L_1'$ and $L_2'$ are conjugate by an element in $G'(F)$.) More precisely, the spaces $\mathcal A_{\mathcal J,\{P'\}}(G')$ consist exactly of those automorphic forms $f\in\mathcal A_{\mathcal J}(G')$, which are negligible along every parabolic $F$-subgroup $Q'\notin\{P'\}$, i.e., with respect to a Levi-decomposition of $Q'=L_{Q'}N_{Q'}$, the constant term $f_{Q'}$ is orthogonal to the space of cuspidal automorphic forms on $L_{Q'}(\A)$.

We remark that within the direct sum \eqref{eq:autdecP'}, the subspace $\mathcal A_{cusp, \mathcal J}(G')$ of all cuspidal automorphic forms in $\mathcal A_\mathcal J(G')$ is given as the summand index by the class $\{G'\}$ itself:
$$\mathcal A_{cusp, \mathcal J}(G')=\mathcal A_{\mathcal J,\{G'\}}(G').$$

\subsection{}
The various summands $\mathcal A_{\mathcal J,\{P'\}}(G')$ can be decomposed even further. Therefore, let $P'=L'N'$ be a Levi-decomposition of the parabolic $F$-subgroup $P'$. Its Levi factor $L'$ is hence of the form $L'\cong\prod_{i=1}^r GL_{m_i}'$, with $\sum_{i=1}^r m_i=m$. Now, recall from \cite{schwfr}, 1.2, the notion of an {\it associate class} $\varphi_{P'}$ of cuspidal automorphic representations of the Levi subgroups of the elements in the class $\{P'\}$. These classes $\varphi_{P'}$ may be parameterized by pairs of the form $(\chi,\tilde\Pi')$, where
\begin{enumerate}
\item $\tilde\Pi'$ is a unitary cuspidal automorphic representation of $L'(\A)$, whose central character vanishes on the diagonally embedded group $\R_+^r\hookrightarrow L'_\infty\hookrightarrow L'(\A)\cong\prod_{i=1}^r GL_{m_i}'(\A)$
\item $\chi:\R_+^r\rightarrow\C^*$ is a Lie group character and
\item the infinitesimal character of $\tilde\Pi'_\infty$ and the derivative $d\chi\in\textrm{Hom}(\C^r,\C)$ of $\chi$ are compatible with the action of $\mathcal J$ (cf. \cite{schwfr}, 1.2).
\end{enumerate}
Each associate class $\varphi_{P'}$ may hence be represented by a cuspidal automorphic representation $\Pi':=\tilde\Pi'\otimes e^{\langle d\chi,H_{P'}\rangle}$ of $L'(\A)$, where $H_{P'}$ is again the Harish-Chandra homomorphism $L'(\A)\ra\C^r$, cf. \cite{franke}, p. 185. Given $\varphi_{P'}$, represented by a cuspidal representation $\Pi'$ of the above form, a $G'(\A)$-submodule
$$\mathcal A_{\mathcal J,\{P'\},\varphi_{P'}}(G')$$
of $\mathcal A_{\mathcal J,\{P'\}}(G')$ was defined in \cite{schwfr}, 1.3 as follows: It is the span of all possible holomorphic values or residues of all Eisenstein series attached to $\tilde\Pi'$, evaluated at the point $d\chi$, together with all their derivatives. This definition is independent of the choice of the representatives $P'$ and $\Pi'$, thanks to the functional equations satisfied by the Eisenstein series considered. For details, we refer the reader to \cite{schwfr} 1.2-1.4.

As a consequence of Franke's theorem, cf. \cite{franke}, Thm.\ 14, the following refined decomposition as $G'(\A)$-modules of the spaces $\mathcal A_{\mathcal J,\{P'\}}(G')$ of automorphic forms was obtained in \cite{schwfr}, Thm.\ 1.4:
\begin{equation}\label{eq:autdecP'phi}
\mathcal A_{\mathcal J,\{P'\}}(G')\cong\bigoplus_{\varphi_{P'}} \mathcal A_{\mathcal J,\{P'\},\varphi_{P'}}(G').
\end{equation}

This gives rise to the following

\begin{defn}
Let $\Psi'$ be an automorphic representation of $G'(\A)$, whose central character is trivial on the diagonally embedded group $\R_+$. If $\Psi'$ is an irreducible subquotient of the space $\mathcal A_{\mathcal J,\{P'\},\varphi_{P'}}(G')$, we call the associate class $\{P'\}$ a {\it parabolic support} and the associate class $\varphi_{P'}$ a {\it cuspidal support} of $\Psi'$.
\end{defn}

\subsection{Discrete series representations} Let $Z'/F$ be the center of the algebraic group $G'/F$ and
$$\omega: Z'(F)\backslash Z'(\A)\rightarrow\C^*$$
be a unitary, smooth character. The space
$$L^2_{dis}(G'(F)\backslash G'(\A),\omega)$$
denotes the space of all automorphic functions
$$f:G'(F)\backslash G'(\A)\rightarrow\C$$
which satisfy $f(zg)=\omega(z)f(g)$, for all $z\in Z'(\A)$ and almost all $g\in G'(\A)$, and $|f|^2$ is square-integrable as a function on $Z'(\A)G'(F)\backslash G'(\A)$ (with respect to the usual quotient measure). This space is the {\it discrete spectrum} of $G'(\A)$, cf. Borel \cite{bor3}, 9.6. Via the right regular action, it is a representation space of $G'(\A)$ and it decomposes as a direct Hilbert sum of irreducible, unitary $G'(\A)$-representations $\tilde\Pi'$, called {\it unitary discrete series} representations:
$$L^2_{dis}(G'(F)\backslash G'(\A),\omega)=\widehat\bigoplus\tilde\Pi'.$$
Observe that there is no term for the multiplicity appearing in this decomposition, which is due to the Multiplicity One Theorem for unitary discrete series, proved by Badulescu--Renard in \cite{ioanrenard}, Thm.\ 18.1.(b).

\subsection{}
Let us now define a certain sub-class of automorphic representations, which we will mainly focus on in this paper.
\begin{defn}
Let $\mathscr D(G')$ be the family of all twists
$$\Pi'=\tilde\Pi'\otimes|\textrm{det}'|^s,$$
$\tilde\Pi'$ being a unitary discrete series representation of $G'(\A)$ and $s\in\C$. These twisted representations $\Pi'$ are usually called {\it essentially discrete series} representations, but for brevity of terminology we will only say {\it discrete series} representations.
\end{defn}

Clearly, $\mathscr D(G')$ contains the family of all twists of unitary cuspidal automorphic representations. In accordance with our previous terminology, we will henceforth call such a twist simply a {\it cuspidal} automorphic representation.

\section{A classification of the cohomological, irreducible, unitary dual of $GL_k(\H)$ and $GL_n(\R)$}\label{sect:irrunitcoh}
\subsection{}
Being interested in arithmetic properties of automorphic forms of $G'/F$, the $(\g'_\infty,K'^\circ_\infty)$-cohomology of automorphic representations of $G'(\A)$ will be an important tool. As we are planning to compare cohomological discrete series representations of the inner form $G'(\A)$ to cohomological automorphic representations of $G(\A)$, we shall in particular obtain some knowledge on the cohomological representations of the groups $G'_v$ and $G_v$, $v$ being a non-split archimedean place. In this section, we will hence determine the cohomological, irreducible, unitary dual of the real Lie groups $G'_v$ and $G_v$, $v\in V_\infty$ non-split. This amounts in giving a classification of all irreducible unitary representations of $G'_v=GL_k(\H)$, $k\geq 1$, and $G_v=GL_n(\R)$, $n=2k$, which have non-zero $(\g'_v,K'^\circ_v)$-cohomology twisted by some irreducible highest weight representation $E_{\mu_v}$.

Speh gave a list of irreducible unitary representations of $GL_n(\R)$ in \cite{speh}, which have non-zero cohomology with respect to the trivial representation $E_{\mu_v}=\C$. Using translation functors, see, e.g. \cite{bowa} VI, sect. 0, this result can be adopted to the case of a general coefficient system $E_{\mu_v}$. However, as we also want to classify the cohomological, irreducible, unitary dual of $GL_k(\H)$, we are not going to use Speh's paper, but directly use the results of Vogan-Zuckerman, \cite{vozu}, in order to give a complete classification of the cohomological, irreducible, unitary dual of $GL_k(\H)$ and $GL_n(\R)$ in one go. To lighten the burden on the notation, we henceforth drag the subscript ``$v$'' about, keeping in mind that all objects are local ones at a non-split archimedean place.

\subsection{Non-equivalent $\theta'$-stable parabolic subalgebras of $\g'$}\label{sect:parabg'}
Let $\theta'$ be the usual Cartan-involution $\theta'(X)=-\bar X^t$ on $\g'$ ($\bar X$ denoting the standard conjugation of quaternionic matrices) giving rise to the Cartan decomposition $\g'\cong\k'\oplus\p'$\footnote{It is traditional to denote the compact part within the Cartan decomposition by $\k$ and so we stick to it; our $\k'$ in this subsection should, however, not be confused with the Lie algebra of $K'=Sp(k)\R^*$. The same remark applies to $\k$ and $K=O(n)\R^*$ in the next subsection}. According to this decomposition, let $\h'=\t'\oplus\a'$ be a maximal compact, $\theta'$-stable Cartan subalgebra. We take

$$\t'=\left\{\left( \begin{array}{ccc}
ix_1 &  & 0\\
  &\ddots &\\
0 &  & ix_k
\end{array}\right)\Bigg| x_i\in\R\right\} \quad\quad \a'=\left\{\left( \begin{array}{ccc}
y_1 &  & 0\\
  &\ddots &\\
0 &  & y_k
\end{array}\right)\Bigg| y_i\in\R\right\}.$$
For any $H\in\h'_\C$, let $e_j(H):=ix_j$ and $f_j(H):=y_j$, $1\leq j\leq k$. Then the set of roots of $\g'_\C$ with respect to $\h'_\C$ is given by
$$\Delta(\g'_\C,\h'_\C)=\{\pm e_i\pm e_j\pm (f_i-f_j), 1\leq i< j\leq k\}\cup\{\pm 2e_i, 1\leq i\leq k\}.$$
Furthermore,
\begin{eqnarray*}
\Delta(\k'_\C,\t'_\C) & = & \{\pm e_i\pm e_j, 1\leq i< j\leq k\}\cup\{\pm 2e_i, 1\leq i\leq k\}\\
\Delta(\p'_\C,\t'_\C) & = & \{\pm e_i\pm e_j, 1\leq i< j\leq k\}\\
W_{\k'}& \cong & S_k\ltimes \{\pm 1\}^k,
\end{eqnarray*}
where $S_k$ is the symmetric group of index $k$. It is hence clear that an element of the Weyl group $W_{\k'}$ acts on $H\in i\t'$ by permuting the entries $x_j$, $1\leq j\leq k$ and changing their signs. This gives rise to the quotient $i\t'/W_{\k'}$, which admits a polyhedral cone decomposition. We obtain

\begin{lem}\label{lem:g'1}
The following sets are pairwise in $1$-to-$1$ correspondence:
\begin{itemize}
\item[(1)] $\{$open, polyhedral cones in $i\t'/W_{\k'}\}$
\item[(2)] $\{ x=(\underbrace{0,...,0}_{k_0},\underbrace{1,...,1}_{k_1},...,\underbrace{r,...,r}_{k_r}), k=\sum_{i=0}^rk_i, k_i>0 \textrm{ for }i\geq 1\}$
\item[(3)] $\{$ordered partitions of $k$: $\underline k = [k_0,k_1,...,k_r], k_0\geq 0, k_i>0 \textrm{ for }i\geq 1\}$
\end{itemize}
\end{lem}
\begin{proof}
The existence of the bijection (1)$\leftrightarrow$(2) is a direct consequence of the description of the action of $W_{\k'}$ on $i\t'$. (The $j$-th entry of $x$ is its coordinate with respect to the basis element of $i\t'$, which is dual to $e_{k-j+1}$.) The bijection (2)$\leftrightarrow$(3) holds trivially.
\end{proof}

Now, let $\mathfrak Q'$ be the set of all $\theta'$-stable parabolic subalgebras $\q'$ of $\g'_\C$, which contain a fixed Borel subalgebra. Recall that $K'\cong Sp(k)\R^*$ acts on this set by the adjoint action. This leads to a finite set $\mathfrak Q'/K'$ of orbits or - otherwise put - $K'$-conjugacy classes of $\theta'$-stable parabolic subalgebras $\q'$ containing a fixed Borel subalgebra. The following lemma says that the set of open polyhedral cones in $i\t'/W_{\k'}$ and $\mathfrak Q'/K'$ are in bijection.

\begin{lem}[\cite{knappvogan} IV, Prop.\ 4.76]\label{lem:g'2}
Every $x\in i\t'/W_{\k'}$ defines a $\theta'$-stable,
parabolic subalgebra $\q'_x$ in $\g'_\C$ by assigning it a Levi decomposition $\q'_x:=\l'_x\oplus\u'_x,$ where
$$\l'_x:=\h'_\C\oplus\bigoplus_{\substack{\alpha\in\Delta(\g'_\C,\t'_\C)\\ \alpha(x)=0}}(\g'_\C)_\alpha\quad\quad\textrm{and}\quad\quad \u'_x:=\bigoplus_{\substack{\alpha\in\Delta(\g'_\C,\t'_\C)\\ \alpha(x)>0}}(\g'_\C)_\alpha$$
and two such $x_1$, $x_2$ define the same parabolic subalgebra $\q'_{x_1}=\q'_{x_2}$ if and only if they are
in the same open, polyhedral cone. \\ Conversely, up to conjugacy by $K'$, any
$\theta'$-stable, parabolic subalgebra $\q'$ containing a fixed Borel subalgebra is of the form
$\q=\q_x$, with $x\in i\t'/W_{\k'}$.
\end{lem}

Together with Lem.\ \ref{lem:g'1}, the last lemma provides a parametrization of the set of $K'$-conjugacy classes of $\theta'$-stable parabolic subalgebras $\q'$ of $\g'_\C$ containing a fixed Borel subalgebra by ordered partitions $\underline k=[k_0,k_1,...,k_r]$ of $k$: $k=\sum_{i=0}^r k_i$, $k_0\geq 0$ and $k_i>0$ for $i\geq 1$. Therefore, we shall henceforth write $\q'_{\underline k}$ for a class in $\mathfrak Q'/K'$. Its Levi subalgebra has real part
\begin{equation}\label{eq:lreal'}
\l'^\R_{\underline k}:=\l'_{\underline k}\cap\g'\cong\g\l_{k_0}(\H)\oplus\bigoplus_{i=1}^r\g\l_{k_i}(\C).
\end{equation}
In view of Salamanca-Riba's paper \cite{salam}, Prop.\ 1.11, we shall introduce another relation on $\mathfrak Q'/K'$:
$\q'_{\underline{k}_1}\sim\q'_{\underline{k}_2}$ if and only if $\Delta(\u'_{\underline{k}_1}\cap\p'_\C,\t'_\C)=\Delta(\u'_{\underline{k}_2}\cap\p'_\C,\t'_\C)$. It is an easy exercise to check that in terms of the parameterizing partitions ${\underline{k}_1}$ and ${\underline{k}_2}$, $\q'_{\underline{k}_1}\sim\q'_{\underline{k}_2}$ if and only if either $\underline{k}_1=\underline{k}_2$ or
\begin{equation}\label{eq:equivrel}
\underline k_1=[0,1,k_2,...,k_r]\quad\quad\textrm{and}\quad\quad\underline k_2=[1,k_2,...,k_r].
\end{equation}
(E.g., $\underline k_1=[0,1,2,1,3]$ and $\underline k_2=[1,2,1,3]$ for $k=7$.) Within such a non-singleton equivalence class, we pick the $K'$-conjugacy class parameterized by $\underline k_1$. To conclude this subsection, we define $\mathcal Q'$ to be the so chosen set of representatives of equivalence classes in $\mathfrak Q'/K'$, i.e., of $K'$-conjugacy classes of $\theta'$-stable parabolic subalgebras $\q'=\q'_{\underline k}$ of $\g'_\C$ containing a fixed Borel subalgebra.

\subsection{Non-equivalent $\theta$-stable parabolic subalgebras of $\g$}\label{sect:parabg}
We will now determine the same data for the split case. Therefore, let $\theta$ be the usual Cartan-involution $\theta(X)=- X^t$ on $\g$ leading to the Cartan decomposition $\g\cong\k\oplus\p$ and standard maximal compact, $\theta$-stable Cartan subalgebra $\h=\t\oplus\a$. We have

$$\t=\left\{\left( \begin{array}{ccccc}
0 &  x_1 & & & 0\\
-x_1 & 0  & & &\\
& & \ddots & &\\
 &  & & 0 & x_k\\
 &  & & -x_k & 0
\end{array}\right)\Bigg| x_i\in\R\right\}
$$
and
$$
\a=\left\{\left( \begin{array}{ccccc}
y_1 &  & & & \\
& y_1  & &\\
& &\ddots & &\\
 &&  & y_k & \\
& &  &  & y_k
\end{array}\right)\Bigg| y_i\in\R\right\}.$$
For $H\in\h_\C$, let again $e_j(H):=ix_j$ and $f_j(H):=y_j$, $1\leq j\leq k$. Then the set of roots of $\g_\C$ with respect to $\h_\C$ is given by
$$\Delta(\g_\C,\h_\C)=\{\pm e_i\pm e_j\pm (f_i-f_j), 1\leq i< j\leq k\}\cup\{\pm 2e_i, 1\leq i\leq k\}.$$
However, in contrast to the non-split case, the sets of compact and non-compact roots change their roles
\begin{eqnarray*}
\Delta(\k_\C,\t_\C) & = & \{\pm e_i\pm e_j, 1\leq i< j\leq k\}\\
\Delta(\p_\C,\t_\C) & = & \{\pm e_i\pm e_j, 1\leq i< j\leq k\}\cup\{\pm 2e_i, 1\leq i\leq k\}.
\end{eqnarray*}
Slightly abusing notation, we set
$$W_{\k}:= S_k\ltimes \{\pm 1\}^k,$$
and let this ``extended'' Weyl group $W_{\k}$ act on $H\in i\t$ by permuting the entries $x_j$, $1\leq j\leq k$ and changing their signs. Whence, the following lemma is obvious by Lem.\ \ref{lem:g'1}. See also Speh's aforementioned article \cite{speh}, p. 464.

\begin{lem}\label{lem:g1}
The following sets are pairwise in $1$-to-$1$ correspondence:
\begin{itemize}
\item[(1)] $\{$open, polyhedral cones in $i\t/W_{\k}\}$
\item[(2)] $\{ x=(\underbrace{0,...,0}_{k_0},\underbrace{1,...,1}_{k_1},...,\underbrace{r,...,r}_{k_r}), k=\sum_{i=0}^rk_i, k_i>0 \textrm{ for }i\geq 1\}$
\item[(3)] $\{$ordered partitions of $n$: $\underline n = [n_0,n_1,...,n_r], n_0\geq 0, n_i=2k_i$, where $k_i>0$ for $i\geq 1\}$
\end{itemize}
\end{lem}

Now, let $\mathfrak Q$ be the set of all $\theta$-stable parabolic subalgebras $\q$ of $\g_\C$, which contain a fixed Borel subalgebra. Recall that $K^\circ\cong SO(n)\R_+$ acts on this set by the adjoint action. This leads a finite set $\mathfrak Q/K^\circ$ of $K^\circ$-conjugacy classes of $\theta$-stable parabolic subalgebras $\q$ containing a fixed Borel subalgebra. Again there is a bijection between the set of open polyhedral cones in $i\t/W_{\k}$ and $\mathfrak Q/K^\circ$, explaining our choice for $W_\k$:

\begin{lem}[\cite{knappvogan} IV, Prop.\ 4.76]\label{lem:g2}
Every $x\in i\t/W_{\k}$ defines a $\theta$-stable,
parabolic subalgebra $\q_x$ in $\g_\C$ via $\q_x:=\l_x\oplus\u_x,$ where
$$\l_x:=\h_\C\oplus\bigoplus_{\substack{\alpha\in\Delta(\g_\C,\t_\C)\\ \alpha(x)=0}}(\g_\C)_\alpha\quad\quad\textrm{and}\quad\quad \u_x:=\bigoplus_{\substack{\alpha\in\Delta(\g_\C,\t_\C)\\ \alpha(x)>0}}(\g_\C)_\alpha$$
and two such $x_1$, $x_2$ define the same parabolic subalgebra $\q_{x_1}=\q_{x_2}$ if and only if they are
in the same open, polyhedral cone. \\ Conversely, up to conjugacy by $K^\circ$, any
$\theta$-stable, parabolic subalgebra $\q$ containing a fixed Borel subalgebra of $\g_\C$ is of the form
$\q=\q_x$, with $x\in i\t/W_{\k}$.
\end{lem}

In particular, Lem.\ \ref{lem:g1} provides a parametrization of the set of $K^\circ$-conjugacy classes of $\theta$-stable parabolic subalgebras $\q$ of $\g_\C$ containing a fixed Borel subalgebra by ordered partitions $\underline n=[n_0,n_1,...,n_r]$ of $n$: $n=\sum_{i=0}^r n_i$, $n_0\geq 0$ and $n_i=2k_i>0$ for $i\geq 1$. Therefore, we shall henceforth write $\q_{\underline n}$ for a class in $\mathfrak Q/K^\circ$. The real part of its Levi subalgebra is isomorphic to
\begin{equation}\label{eq:lreal}
\l^\R_{\underline n}:=\l_{\underline n}\cap\g\cong\g\l_{n_0}(\R)\oplus\bigoplus_{i=1}^r\g\l_{k_i}(\C),
\end{equation}
see also \cite{speh}, p. 464. Introducing the analogous equivalence relation on $\mathfrak Q/K^\circ$ by letting $\q_{\underline{n}_1}\sim\q_{\underline{n}_2}$ if and only if $\Delta(\u_{\underline{n}_1}\cap\p_\C,\t_\C)=\Delta(\u_{\underline{n}_2}\cap\p_\C,\t_\C)$, we encounter a different phenomenon than in the non-split case. As every root appears at least once in the set of non-compact roots, for any $\theta$-stable parabolic subalgebra $\q_{\underline n}$ of $\g_\C$
$$\Delta(\u_{\underline{n}}\cap\p_\C,\t_\C)=\Delta(\u_{\underline{n}},\t_\C)$$
and so the relation ``$\sim$'' degenerates to equality:
$$\q_{\underline{n}_1}\sim\q_{\underline{n}_2}\quad\Leftrightarrow\quad \q_{\underline{n}_1}=\q_{\underline{n}_2}.$$
Again, we denote by $\mathcal Q$ the set of all such equivalence classes, i.e., $\mathcal Q=\mathfrak Q/K^\circ$ itself in this case.

\subsection{}\label{sect:L&L'}
Having compiled this data for $\g'$ and $\g$, we can now give the desired classification of the cohomological, irreducible, unitary dual of $G'$ and $G$. Therefore, let $E$ be an irreducible representation of $G'$ or $G$ on a finite-dimensional complex vector space. Observe that we have changed Cartan subalgebras in this section, so we will write $\lambda$ for its highest weight with respect to the standard choice of positivity on $\Delta(\g_\C,\h_\C)=\Delta(\g'_\C,\h'_\C)$, rather than $\mu$. As we will soon use the results of Vogan--Zuckerman \cite{vozu}, Section 5, we shall check when $\lambda$ does define a so-called {\it admissible} character: Given $\q'\in\mathcal Q'$ (resp., $\q\in\mathcal Q$) let $L'\subseteq G'$ (resp., $L\subseteq G$) be the connected subgroup with Lie algebra $\l'^\R$ (resp., $\l^\R$). According to \cite{vozu}, (5.1), a linear functional $\xi$ on $\l'$ (resp., $\l$) is called admissible if
\begin{enumerate}
\item $\xi$ is the differential of a unitary character of $L'$ (resp., $L$)
\item $\langle\xi,\alpha\rangle\geq 0$ for all roots $\alpha$ appearing in $\u'$ (resp., $\u$).
\end{enumerate}
As a corollary, a highest weight $\lambda$ defines an admissible character of $\l'$ if and only if
\begin{equation}\label{eq:extending}
\lambda|_{[\l'^\R,\l'^\R]}=0\quad\textrm{and}\quad\lambda|_{\a'}=0.
\end{equation}
As an explanation, the vanishing of $\lambda$ on the commutator $[\l'^\R,\l'^\R]$ is equivalent to saying that $\lambda:\h'_\C\ra\C$ can be extended to a character on $\l'\supseteq\h'_\C$, while the second condition $\lambda|_{\a'}=0$ is equivalent to $\lambda$ coming from a {\it unitary} character. The latter condition can once more be equivalently reformulated by saying that $\lambda\circ\theta'=\lambda$ or - again equivalent - that $E$ is self-dual. In particular, we may view $\lambda$ as being expressed in the functionals $e_j$, $1\leq j\leq k$, only, writing $\lambda=(\lambda_1,...,\lambda_k)$, $\lambda_j$ being the coefficient of $e_j$ in such a decomposition. Clearly, the same statements hold for $\g'$ being replaced by $\g$, removing the prime everywhere.

We let $\mathcal Q'(\lambda)$ (resp., $\mathcal Q(\lambda)$) be the set of all $\q'\in\mathcal Q'$ (resp., $\q\in\mathcal Q$), to whose Levi parts $\l'$
(resp., $\l$) a given highest weight $\lambda$ can be admissibly extended.

\subsection{The cohomological, irreducible, unitary dual of $G'=GL_k(\H)$}
Now, we may state the following theorem settling the non-split case.

\begin{thm}\label{prop:irrunitG'}
Let $G'=GL_k(\H)$, $k\geq 1$ and $E=E_\lambda$ a highest weight representation. The following holds:
\begin{enumerate}
\item To each $\q'=\q'_{\underline k}\in\mathcal Q'(\lambda)$, there exists a unique irreducible unitary $G'$-representation $A_{\q'}(\lambda)$.
None of these are pairwise isomorphic, i.e., $A_{\q'_{\underline{k}_1}}(\lambda)$$\cong A_{\q'_{\underline{k}_2}}(\lambda)$ if and only if $\q'_{\underline{k}_1}=\q'_{\underline{k}_2}$.
\item Each $A_{\q'}(\lambda)$ is cohomological with respect to $E_\lambda$, i.e.,
$$H^q(\g',K'^\circ,A_{\q'}(\lambda)\otimes E_\lambda)\neq 0$$
for some degree $q\geq 0$, and all irreducible unitary $G'$-modules which are cohomological with respect to $E_\lambda$ are obtained in this way.
\item The Poincar\'e-polynomial of the cohomology ring $H^*(\g',K'^\circ,A_{\q'(\lambda)}\otimes E_\lambda)$ is independent of $\lambda$ and given by
$$P(\underline k,X)=\frac{X^{\dim(\u'_{\underline k}\cap\p'_\C)}}{1+X}\prod_{i=1}^r\prod_{j=1}^{k_i}(1+X^{2j-1})\prod_{j=1}^{k_0}(1+X^{4j-3}).$$
\item Among all $A_{\q'}(\lambda)$, $\q'=\q'_{\underline k}\in\mathcal Q'(\lambda)$, the representation indexed by $\q'_{\underline k}$ with $\underline k=[0,1,1,...,1]$ is the only tempered representation.
\end{enumerate}
\end{thm}
\begin{proof}
(1): The existence of the irreducible $G'$-module $A_{\q'}(\lambda)$ is proved by Vogan--Zuckerman in \cite{vozu}, Thm.\ 5.3 together with (5.1). Its unitarity is shown by Vogan in \cite{vogan}, Thm.\ 1.3, while their pairwise inequivalence is a consequence of the work of Salamanca-Riba \cite{salam}, Prop.\ 1.11.\\
(2): This is \cite{vozu}, Thm.\ 5.5 and 5.6.\\
(3): By \cite{vozu}, Thm.\ 5.5 we know that
$$H^q(\g',K'^\circ,A_{\q'_{\underline k}}(\lambda)\otimes E_\lambda)\cong H^{q-\dim(\u'_{\underline k}\cap\p'_\C)}(\l'^\R_{\underline k},\l'^\R_{\underline k}\cap(\R\oplus\mathfrak{sp}(k)),\C),$$
which is already independent of $\lambda$ but only dependent on the partition $\underline k$ determining $\q'=\q'_{\underline k}$. Using the concrete form of $\l'^\R_{\underline k}$, given in \eqref{eq:lreal'} and the fact that $(\g',K'^\circ)$-cohomology satisfies the K\"unneth-rule, it is now an easy exercise to calculate the Poincar\'e-polynomial of the cohomology ring $H^*(\g',K'^\circ,A_{\q'}(\lambda)\otimes E_\lambda)$.\\
(4): Taking into account our description of the real part of the Levi subalgebra of a $\theta'$-stable parabolic $\q'_{\underline k}\in\mathcal Q'(\lambda)$ given in \eqref{eq:lreal'}, this follows from \cite{vozu}, last paragraph on p. 58.
\end{proof}

\begin{rem}
It is to avoid redundancies in the list of $A_{\q'}(\lambda)$-modules, that we introduced the equivalence relation ``$\sim$'' on $\mathfrak Q'/K'$. If there is a non-trivial equivalence $\q'_{\underline{k}_1}\sim\q'_{\underline{k}_2}$, i.e., $\underline k_1=[0,1,k_2,...,k_r]$ and $\underline k_2=[1,k_2,...,k_r]$, and $\lambda$ can be extended to an admissible character of $\l'_{\underline k_2}$, too, then $A_{\q'_{\underline k_2}}(\lambda)$ exists. However, it such a case
$$A_{\q'_{\underline{k}_1}}(\lambda)\cong A_{\q'_{\underline{k}_2}}(\lambda),$$
cf. \cite{salam}, Prop.\ 1.11.
\end{rem}

Thm.\ \ref{prop:irrunitG'} motivates the following definition.

\begin{defn}
Let $E=E_\mu$ be a highest weight representation as in Section \ref{sect:highweights}, whose highest weight with respect to $\h'_\C$ is $\lambda$. We denote
$$\textrm{Coh}_\mu(G'):=\{A_{\q'}(\lambda), \q'=\q'_{\underline k}\in\mathcal Q'(\lambda)\}.$$
According to Thm.\ \ref{prop:irrunitG'}, Coh$_\mu(G')$ is the set of all irreducible, unitary $G'$-representations, which have non-zero $(\g',K'^\circ)$-cohomology with respect to $E$.
\end{defn}

\subsection{The cohomological, irreducible, unitary dual of $G=GL_n(\R)$}
The split case is slightly more complicated, as $G=GL_n(\R)$ is not connected. Still, we obtain the following theorem.

\begin{thm}\label{prop:irrunitG}
Let $G=GL_n(\R)$, $n=2k\geq 2$ and $E=E_\lambda$ a highest weight representation. The following holds:
\begin{enumerate}
\item To each $\q=\q_{\underline n}\in\mathcal Q(\lambda)$, there exists a unique irreducible unitary $G$-representation $A_{\q}(\lambda)$. None of these are pairwise isomorphic, i.e., $A_{\q_{\underline{n}_1}}(\lambda)$$\cong A_{\q_{\underline{n}_2}}(\lambda)$ if and only if $\q_{\underline{n}_1}=\q_{\underline{n}_2}$.
\item Each $A_{\q}(\lambda)$ is cohomological with respect to $E_\lambda$, i.e.,
$$H^q(\g,K^\circ,A_{\q}(\lambda)\otimes E_\lambda)\neq 0$$
for some degree $q\geq 0$. Conversely, all irreducible unitary $G$-modules which are cohomological with respect to $E_\lambda$ are of the form
$$A_{\q}(\lambda)\otimes\sgn^\varepsilon, \quad\quad\q=\q_{\underline n}\in\mathcal Q(\lambda), \varepsilon\in\{0,1\}.$$
\item Among all $A_{\q}(\lambda)$, $\q=\q_{\underline n}\in\mathcal Q(\lambda)$, the representation indexed by $\q_{\underline n}$ with $\underline n=[0,2,2,...,2]$ is the only tempered representation.
\end{enumerate}
\end{thm}
\begin{proof}
(1): The existence of the irreducible $G'$-module $A_{\q'}(\lambda)$ is again a consequence \cite{vozu}, Thm.\ 5.3 together with (5.1). Its unitarity is shown in \cite{vogan}, Thm.\ 1.3, while their pairwise inequivalence is a consequence of \cite{salam}, Prop.\ 1.11.\\
(2): This is \cite{vozu}, Thm.\ 5.5 and 5.6., where the twist by $\sgn^\varepsilon$ takes account of the disconnectedness of $G$.\\
(3): Recalling our concrete description of the real part of the Levi subalgebra of a $\theta$-stable parabolic $\q_{\underline n}\in\mathcal Q(\lambda)$ given in \eqref{eq:lreal}, this follows again from \cite{vozu}, last paragraph on p. 58.
\end{proof}

\begin{defn}
\label{def:lambda-mu}
Let $E=E_\mu$ be a highest weight representation as in Section \ref{sect:highweights}, whose highest weight with respect to $\h_\C$ is $\lambda$. We denote
$$\textrm{Coh}_\mu(G):=\{A_{\q}(\lambda)\otimes\sgn^\varepsilon, \q=\q_{\underline n}\in\mathcal Q(\lambda),\varepsilon\in\{0,1\}\}.$$
According to Thm.\ \ref{prop:irrunitG}, Coh$_\mu(G)$ is the set of all irreducible, unitary $G$-representations, which have non-zero $(\g,K^\circ)$-cohomology with respect to $E$.
\end{defn}

\section{The local and global Jacquet-Langlands transfer and its interplay with cohomology}
\label{sect:local-global-jl}

\subsection{The local Jacquet-Langlands map}\label{sect:localJL}
In \cite{jaclan} Jacquet--Langlands established a bijection between irreducible representations of $G'_v$ and irreducible, admissible square-integrable representations of $G_v$ in the special case when $m=1$ and $d=2$, i.e., when $G'$ is the group of invertible elements in a quaternion division algebra. (See also Gelbart--Jacquet \cite{gelbjacquet} Thm.\ 8.1 and its generalization of Rogwaski, \cite{rogaw}.) This bijection was extended to a bijection between irreducible, admissible square-integrable representations of $G'_v$ and irreducible, admissible square-integrable representations of $G_v$ for any $m$ and $d$ by Deligne--Kazhdan--Vign\'eras in \cite{dkv} -- the local Jacquet-Langlands correspondence. Finally, Badulescu in \cite{ioan}, p. 406 together with Badulescu--Renard in \cite{ioanrenard}, sect. 13, ``thickened'' the Jacquet-Langlands correspondence to a map
$$|LJ|_v: \tilde U_{cp}(G_v)\rightarrow \tilde U(G'_v)$$
from the family $\tilde U_{cp}(G_v)$ of so-called {\it $d_v$-compatible} irreducible unitary representations of $G_v$ to the family $\tilde U(G'_v)$ of irreducible unitary representations of $G'_v$. (For the notion of $d_v$-compatibility we refer to \cite{ioan}, sect. 15) This map is neither injective nor surjective in general, but restricts to the Jacquet-Langlands correspondence of \cite{dkv} (and \cite{jaclan}) on unitary square-integrable representations. If $d_v=1$, i.e., $G'_v=G_v$ the map $|LJ|_v$ is the identity. Furthermore, $|LJ|_v$ commutes with parabolic induction and forming tensor products. We denote by $U_{cp}(G_v)$ the set of all $\Pi_v=\tilde\Pi_v\otimes|\!\det\!|^s$, with $s\in\C$ and $\tilde\Pi_v\in\tilde U_{cp}(G_v)$, and define
$$|LJ|_v(\Pi_v)=|LJ|_v(\tilde\Pi_v)\otimes\textrm{det}'^s.$$
The analogous notation $U(G'_v)$ is used in the non-split case, i.e., $\Pi'_v\in U(G'_v)$ if and only if $\Pi'_v=\tilde\Pi'_v\otimes|\!\det'\!|^s$ with $\tilde\Pi'_v\in\tilde U(G'_v)$. Furthermore we set for $\Pi_\infty=\bigotimes_{v\in V_\infty}\Pi_v\in U_{cp}(G_\infty)=\bigotimes_{v\in V_\infty}U_{cp}(G_v)$
$$|LJ|_\infty(\Pi_\infty)=\bigotimes_{v\in V_\infty}|LJ|_v(\Pi_v).$$

\subsection{}\label{sect:easylemma}
In the sequel, we will need some particular knowledge about the easiest, non-trivial case of $|LJ|_v$, namely if $v\in V_\infty$ is a non-split place and $m=1$ and $d=2$.
For any integer $l\geq 1$ and $u\in\C$ denote by
$$F(u,l):=\textrm{Sym}^{l-1}\C^2\otimes\textrm{det}'^{-u/2},$$
where $\textrm{Sym}^{l-1}\C^2$ is the unique irreducible, unitary representation of $SL_1(\H)$ of dimension $l$. The representations $F(u,l)$ exhaust all irreducible, finite-dimensional representations of $GL_1(\H)$.

Furthermore, we define
$$D(u,l):=D(l)\otimes|\textrm{det}|^{-u/2},$$
where $D(l)$ the unique irreducible, unitary, discrete series representation of $SL_2^\pm(\R)$ of lowest (non-negative) $O(2)$-type $l+1$.
Then we obtain

\begin{lem}\label{lem:localJL}
Let $v\in V_\infty$ be a non-split place of a quaternion division algebra $D$ and $m=1$.
The family $U_{cp}(GL_2(\R))$ consists precisely of the representations
\begin{itemize}
\item[(i)] $\textrm{\emph{sgn}}^\varepsilon|\!\det\!|^s$, $\varepsilon\in\{0,1\}$ and $s\in\C$
\item[(ii)] $D(u,l)$ for some integer $l\geq 1$ and $u\in\C$.
\end{itemize}
while the family $U(GL_1(\H))$ consists precisely of the finite-dimensional representations
\begin{itemize}
\item[(i')] $F(u',l')$ for some integer $l'\geq 1$ and $u'\in\C$.
\end{itemize}
The local map $|LJ|_v$ is given by
$$|LJ|_v(\Pi_v)=\left\{
\begin{array}{ll}
\textrm{\emph{det}}'^s & \textrm{ if $\Pi_v=\textrm{\emph{sgn}}^\varepsilon|\det|^s$}\\
F(u,l) & \textrm{ if $\Pi_v=D(u,l)$.}
\end{array}
\right.$$
In particular, in this special case, $|LJ|_v$ is surjective but not injective.
\end{lem}
\begin{proof}
This is well-known, respectively follows from the description of the local map $|LJ|_v$ in this case, cf. Badulescu--Renard \cite{ioanrenard}, Thm 13.8.
\end{proof}

\subsection{Local Jacquet-Langlands and cohomology}
We are now ready to prove the first main result of this paper. It compares the cohomological, irreducible, unitary duals of $G'_v$ and $G_v$, $v\in V_\infty$ non-split, via the local Jacquet-Langlands map $|LJ|_v$ and will be of particular importance in the following sections. Nevertheless, we believe that it is interesting in its own right. It fits very well with the philosophy to use functoriality in order to get cohomological automorphic representations. The interested reader may find a survey on this latter topic by Raghuram--Shahidi in \cite{ragsha}, Section 5.2.

\begin{thm}\label{thm:irrunitJL}
Let $v\in V_\infty$ be a non-split place, so $G'_v=GL_k(\H)$, $k\geq 1$ and $G_v=GL_n(\R)$, $n=2k$. Furthermore, let $E=E_{\mu_v}$ be any highest weight representation as in Section \ref{sect:highweights}. Then we get:\\
The local Jacquet-Langlands map $|LJ|_v$ defines a surjection
$$\xymatrix{\textrm{\emph{Coh}}_{\mu_v}(G_v)\ar@{->>}[rr]^{|LJ|_v}  & & \textrm{\emph{Coh}}_{\mu_v}(G'_v)}$$
given explicitly by
$$|LJ|_v(A_{\q_{\underline n}}(\lambda)\otimes\sgn^\varepsilon)=A_{\q'_{\underline k}}(\lambda).$$
Here $\underline n=[n_0,n_1,...,n_r]$ and $\underline k=[k_0,k_1,...,k_r]$ with $n_i=2k_i$, $0\leq i\leq r$.\\
In particular, $|LJ|_v$ maps tempered cohomological representations to tempered cohomological representations.
\end{thm}
\begin{proof}
The proof consists in rewriting the $A_{\q'}(\lambda)$- and $A_\q(\lambda)$-modules as quotients of parabolically induced representations, following Vogan--Zuckerman \cite{vozu}, p. 82, and then applying $|LJ|_v$. In what follows we will freely use the notation of \cite{vozu}, p. 82, which is explained there in details and drop the index ``$v$'', as it is clear that all objects are local ones at a non-split, archimedean place.\\
Let us begin with the non-split case, i.e., with $G'=G'_v$ and a representation $A_{\q'_{\underline k}}(\lambda)$ index by $\underline k=[k_0,k_1,...,k_r]$ as in Section \ref{sect:parabg'}. Recall the Levi subgroup $L'=L'_{\underline k}\cong GL_{k_0}(\H)\times\prod_{i=1}^r GL_{k_i}(\C)$ defined in Section \ref{sect:L&L'}. It contains a maximally split $\theta'$-stable Cartan subgroup $T^+A^d$. In fact, for the Lie algebras we may take
$$\t^+=Lie(T^+)=\t'\quad\quad\textrm{and}\quad\quad\a^d=Lie(A^d)=\a',$$
where $\t'$ and $\a'$ are as in Section \ref{sect:parabg'}. As in \cite{vozu}, p. 82, we let $M^dA^d$ be the Langlands decomposition of the centralizer of $A^d$ in $G'=GL_k(\H)$ and obtain
$$M^dA^d\cong\prod_{j=1}^k GL_1(\H),$$
so $M^d\cong\prod_{j=1}^k SL_1(\H)$ and let $P'$ be a parabolic subgroup containing $M^dA^d$ as Levi factor. It is a minimal parabolic subgroup. The character $\nu^d$ defined on \cite{vozu}, p.82, is given as
\begin{eqnarray*}
\nu^d & = & (\rho_{\gl_{k_0}(\H)},\rho_{\gl_{k_1}(\C)},...,\rho_{\gl_{k_r}(\C)})\\
 & =: & \nu'_{\underline k},
\end{eqnarray*}
where $\rho_\bullet$ is the smallest algebraically integral element in the interior of the dominant Weyl chamber of $\bullet$. We still need to describe the representation $\sigma^d$ of $M^d$, defined abstractly on p. 82 of \cite{vozu}. Therefore, we recall three facts. Firstly, as we may assume that $\lambda$ is self-dual, cf. Section \ref{sect:L&L'}, $\lambda|_{\t^+}=\lambda=(\lambda_1,...,\lambda_k)$ in the coordinates given by the functionals $e_j$, cf. Section \ref{sect:parabg'}. Secondly, $\rho^+$ being defined as the half sum of positive roots of $\t^+=\t'$ in $\m^d\cap\l'$ equals, $\rho^+=(0,...,0,1,...,1)$, the ``$1$'s'' starting from the $k-k_0+1$-st entry. Thirdly, the highest weight of $\sigma^d$ with respect to the system $\Delta(\m^d_\C,\t'_\C)$ equals the Harish-Chandra parameter of $\sigma^d$ subtracted by $\rho_{\m^d}=(1,1,...,1)$. Having collected this information, and writing $\rho(\u')=(\rho^{\u'}_1,...,\rho^{\u'}_k)$ for the half-sum of the roots appearing in $\u'=\u'_{\underline k}$, counted with multiplicity, \cite{vozu}, (6.13) tells us that
$$\sigma^d  \cong  \bigotimes_{j=1}^{k-k_0}\textrm{Sym}^{\lambda_j+\rho^{\u'}_j-1}\C^2\otimes \bigotimes_{j=1}^{k_0}\textrm{Sym}^{\lambda_{j+k-k_0}+\rho^{\u'}_{j+k-k_0}}\C^2.$$
However, by Lem.\ \ref{lem:g'1} $\rho^{\u'}_{j+k-k_0}=0$ for $1\leq j\leq k_0$ and by \eqref{eq:extending} necessarily also $\lambda_{j+k-k_0}=0$ for $1\leq j\leq k_0$. So,
\begin{eqnarray*}
\sigma^d & \cong & \bigotimes_{j=1}^{k-k_0}\textrm{Sym}^{\lambda_j+\rho^{\u'}_j-1}\C^2\\
 & = & \bigotimes_{j=1}^{k-k_0} F(0,\lambda_j+\rho^{\u'}_j)\\
 & =: & \sigma'_{\lambda,\underline k},
\end{eqnarray*}
where $F(0,\lambda_j+\rho^{\u'}_j)$ is the notation used in the previous section. By \cite{vozu}, Thm.\ 6.16, $A_{\q'_{\underline k}}(\lambda)$ is isomorphic to the unique irreducible quotient $J(P',\sigma'_{\lambda,\underline k},\nu'_{\underline k})$ of the parabolically induced representation
$$\textrm{Ind}^{G'}_{P'}[\sigma'_{\lambda,\underline k}\otimes\nu'_{\underline k}].$$
Next, we treat the split case, i.e., $G=GL_n(\R)$, $n=2k\geq 2$, and a representation $A_{\q_{\underline n}}(\lambda)$ parameterized by a partition $\underline n=[n_0,n_1,...,n_r]$ as in Section \ref{sect:parabg}. Consider the Levi subgroup $L:=L_{\underline n}\cong GL_{n_0}(\R)\times\prod_{i=1}^r GL_{k_i}(\C)$ containing a maximally split $\theta$-stable Cartan subgroup which we again denote by $T^+A^d$, following \cite{vozu}, p.82. The Lie algebras of its factors satisfy
$$\t^+=Lie(T^+)\hookrightarrow\t\quad\quad\textrm{and}\quad\quad\a\hookrightarrow\a^d=Lie(A^d),$$
where $\t$ and $\a$ are as in Section \ref{sect:parabg}. The Langlands decomposition of the centralizer of $A^d$ in $G=GL_n(\R)$ is given by
$$M^dA^d\cong\prod_{j=1}^{k-k_0} GL_2(\R)\times \prod_{j=1}^{n_0} GL_1(\R),$$
so $M^d\cong\prod_{j=1}^{k-k_0} SL_2^\pm(\R)\times \prod_{j=1}^{n_0}\{\pm1\}$ and we let $P$ be a parabolic subgroup containing $M^dA^d$ as Levi factor. The character $\nu^d$ defined on \cite{vozu}, p. 82, takes the form
\begin{eqnarray*}
\nu^d & = & (\rho_{\gl_{n_0}(\R)},\rho_{\gl_{k_1}(\C)},...,\rho_{\gl_{k_r}(\C)})\\
 & =: & \nu_{\underline n}.
\end{eqnarray*}
Next, we turn to the discrete series representation $\sigma^d$ of $M^d$. According to \cite{vozu}, (6.13) the Harish-Chandra parameter of $\sigma^d$ is
$$\rho^++\lambda|_{\t^+}+\rho(\u_{\underline n}).$$
Observing that $\m^d\cap\l\cong\t^+$, implies $\rho^+=(0,0,...,0)$. Furthermore, $\lambda|_{\t^+}=\lambda=(\lambda_1,...,\lambda_{k-k_0},0,...0)$, so we obtain that $\sigma^d$ has Harish-Chandra parameter $\lambda|_{\t^+}+\rho(\u_{\underline n})$. The lowest $M^d\cap K$-type of $\sigma^d$, cf. \cite{vozu} (6.14), affords that $m\in\prod_{j=1}^{n_0}\{\pm 1\}$ acts by $\sgn(\det(m))^{e(\lambda)}$ for some $e(\lambda)\in\{0,1\}$. If we observe that the roots appearing in $\u_{\underline n}$ and $\u'_{\underline k}$ are the same, by Lemmas\ \ref{lem:g'1} and \ref{lem:g1}, this defines $\sigma^d$ uniquely as
\begin{eqnarray*}
\sigma^d & \cong & \bigotimes_{j=1}^{k-k_0} D(0,\lambda_j+\rho^{\u'}_j)\otimes\sgn^{e(\lambda)}\\
 & =: & \sigma_{\lambda,\underline n}
\end{eqnarray*}
in the notation introduced in the previous section. If we use once more \cite{vozu}, Thm.\ 6.16, we obtain that $A_{\q_{\underline n}}(\lambda)$ is isomorphic to the unique irreducible quotient $J(P,\sigma_{\lambda,\underline n},\nu_{\underline n})$ of the parabolically induced representation
$$\textrm{Ind}^{G}_{P}[\sigma_{\lambda,\underline n}\otimes\nu_{\underline n}].$$
But now, as $|LJ|$ commutes with parabolic induction and forming tensor products, Lem.\ \ref{lem:localJL} together with induction by stages shows that
$$|LJ|(J(P,\sigma_{\lambda,\underline n},\nu_{\underline n}))=J(P',\sigma'_{\lambda,\underline k},\nu'_{\underline k}).$$
But this implies that
$$|LJ|(A_{\q_{\underline n}}(\lambda)\otimes\sgn^\varepsilon)=A_{\q'_{\underline k}}(\lambda),$$
where $\underline n=[n_0,n_1,...,n_r]$ and $\underline k=[k_0,k_1,...,k_r]$ with $n_i=2k_i$, $0\leq i\leq r$. In particular, following the characterization of the cohomological, irreducible, unitary dual of $G'$ and $G$ in our Prop.\ \ref{prop:irrunitG'} and \ref{prop:irrunitG},
$|LJ|$ defines a surjection
$$\xymatrix{\textrm{Coh}_{\mu}(G)\ar@{->>}[rr]^{|LJ|}  & & \textrm{Coh}_{\mu}(G')}$$
for all highest weight representations $E=E_\mu$. Furthermore, by Thm.\ \ref{prop:irrunitG'} (4) and \ref{prop:irrunitG} (3),
$|LJ|$ maps tempered cohomological representations to tempered cohomological representations. This shows the theorem.
\end{proof}

\begin{rem}
The fibers of $|LJ|_v$ in Coh$_\mu(GL_n(\R))$ over a representation $A_{\q'}(\lambda)$ can be explicitly described using \eqref{eq:equivrel}.
\end{rem}

\subsection{An illustrative example}
Let us exemplify Thm.\ \ref{thm:irrunitJL} by letting $k=2$, so $n=4$, and taking $E=\C$ to be the trivial representation, i.e., $\lambda=0$. We obtain the following two tables, where $Q$ is the standard parabolic subgroup of $G=GL_4(\R)$ with Levi factor $GL_2(\R)\times GL_2(\R)$ and $P'$ is as in the proof of Thm.\ \ref{thm:irrunitJL}.

\begin{table}[h!]
\begin{center}
\begin{tabular}{c|c|c}
  $\underline n$ & $A_{\q_{\underline n}}(0)$ & $A_{\q_{\underline n}}(0)\otimes\sgn$\\ \hline
  $[0,2,2]$ & Ind$^G_Q[D(0,3)\otimes D(0,1)]$ & Ind$^G_Q[D(0,3)\otimes D(0,1)]$ \\
  $[2,2]$ &  Ind$^G_Q[D(0,3)\otimes\triv_{GL_2(\R)}]$ & Ind$^G_Q[D(0,3)\otimes\sgn]$\\
  $[0,4]$ & $J(Q,D(-1,2)\otimes D(1,2))$ & $J(Q,D(-1,2)\otimes D(1,2))$\\
  $[4]$ & $\triv_{GL_4(\R)}$ & $\sgn$\\
  \multicolumn{3}{c}{}
\end{tabular}
\caption{Coh$_0(GL_4(\R))$}\label{t:g}
\end{center}
\end{table}

\begin{table}[h!]
\begin{center}
\begin{tabular}{c|c}
  $\underline k$ & $A_{\q'_{\underline k}}(0)$ \\ \hline
  $[0,1,1]$ & Ind$^{G'}_{P'}[F(0,3)\otimes F(0,1)]$  \\
  $[1,1]$ &  Ind$^{G'}_{P'}[F(0,3)\otimes\triv_{GL_1(\H)}]$ \\
  $[0,2]$ & $J(P',F(-1,2)\otimes F(1,2))$ \\
  $[2]$ & $\triv_{GL_2(\H)}$ \\
   \multicolumn{2}{c}{}
\end{tabular}
\caption{Coh$_0(GL_2(\H))$}\label{t:g'}
\end{center}
\end{table}
Since $F(0,1)=\triv_{GL_1(\H)}$, the representations in the first two lines of Table \ref{t:g'} coincide as predicted. Furthermore, the tempered representations are given by the first rows of Table \ref{t:g} and \ref{t:g'}, i.e., by the partitions $[0,2,2]$ and $[0,1,1]$. Applying Lem.\ \ref{lem:localJL} directly to Table \ref{t:g} and noting that Ind$^G_Q[D(0,3)\otimes\triv_{GL_2(\R)}]$ is isomorphic to the irreducible quotient of Ind$^G_P[D(0,3)\otimes\triv_{GL_1(\R)}|\cdot|^{1/2}\otimes\triv_{GL_1(\R)}|\cdot|^{-1/2}]$, where $P$ is as in the proof of Thm.\ \ref{thm:irrunitJL}, reproves the other assertions in this example.

\subsection{Essentially-tempered representations}\label{sect:tempcoh}
There is a particularly important subclass of irreducible, cohomological representations, namely those which are {\it essentially tempered}. Therefore recall that an irreducible admissible representation $\Pi'_v$ of $G'_v$, $v\in V_\infty$, is called {\it essentially tempered}, if it is of the form $\Pi'_v\cong \tilde\Pi'_v\otimes|\!\det'\!|^{s_v}$, with $\tilde\Pi'_v$ being an irreducible, unitary, tempered representation and $s_v\in\C$. A representation $\Pi'_\infty=\bigotimes_{v\in V_\infty} \Pi'_v$ of $G'_\infty$ is analogously called essentially tempered if all its local factors $\Pi'_v$ are. The analog definition applies to representations of $G_v$ when removing the prime.

According to Thm.\ \ref{thm:irrunitJL} and the definition of $|LJ|_v$, a cohomological, irreducible, admissible, essentially tempered representation $\Pi_v$ of $G_v$ is mapped via $|LJ|_v$ onto a cohomological, irreducible, admissible, essentially tempered representation $\Pi'_v$ of $G'_v$. Let us now shortly describe these representations a little bit more explicitly.

\subsubsection*{Non-split considerations}
Let $v\in V_\infty$ be non-split and let $E=E_{\mu_v}$ be a highest weight representation as in Section \ref{sect:highweights}. For sake of simplicity, we will again drop the subscript ``$v$'' in this subsection for all local objects. By the same reason as in Section \ref{sect:L&L'} we see that, in order to admit a cohomological, essentially tempered, irreducible admissible representation $\Pi'$ of $G'$, it is necessary for the coefficient system $E_{\mu}$ to be essentially self-dual. We may hence assume without loss of generality that from now on in this subsection $E_{\mu}$ is essentially self-dual writing $E_{\mu}=E_{\mu}^{\sf v}\otimes\det^{w}$, $w\in\Z$. Otherwise put, if $\mu=(\mu_1,...,\mu_n)$ is the highest weight of $E_\mu$,
$$\mu_{i}+\mu_{n-i+1}=w \quad\textrm{for}\quad 1\leq i\leq n.$$
Let $\ell:=(\ell_1,...,\ell_k)$ with
$$\ell_i:=\mu_i-\mu_{n-i+1}+(n-2i+1)=w-2\mu_{n-i+1}+(n-2i+1).$$
We obtain that $\ell_1>\ell_2>...>\ell_k\geq 1$.  By Thm.\ \ref{prop:irrunitG'} (4) and Borel--Wallach \cite{bowa}, I, Thm.\ 5.3, the unique, irreducible, admissible, essentially tempered representation of $G'$, which has non-trivial $(\g',K'^\circ)$-cohomology twisted by $E_\mu$ is
$$A_{\q'_{\underline k}}(\lambda)\otimes\textrm{det}'^{-w/2},$$
where $\underline k=[0,1,1,...,1]$. It is furthermore a consequence of the proof of Thm.\ \ref{thm:irrunitJL}, that this latter representation is isomorphic to the unique irreducible quotient of
$$\textrm{Ind}^{G'}_{P'}[F(w,\ell_1)\otimes...\otimes F(w,\ell_k)].$$
Now, \cite{hch}, Thm.\ 1, p. 198, implies that this normalized induced representation is already irreducible itself, whence
$$J'(w,\ell):=A_{\q'_{\underline k}}(\lambda)\otimes\textrm{det}'^{-w/2}\cong \textrm{Ind}^{G'}_{P'}[F(w,\ell_1)\otimes...\otimes F(w,\ell_k)]$$
is the unique, irreducible, admissible, essentially tempered representation of $G'$, which has non-trivial $(\g',K'^\circ)$-cohomology twisted by $E_\mu$.
Furthermore,
$$H^q(\mathfrak{gl}_k(\H),Sp(k)\R_+,J'(w,\ell)\otimes E_\mu)\cong\bigwedge^{q-k(k-1)}\C^{k-1},$$
which follows from Thm.\ \ref{prop:irrunitG'} (3) or directly by \cite{bowa} III. Thm 5.1.

The split case $G=G_v=GL_n(\R)$, $n=2k$, is treated in analogy. By Thm.\ \ref{prop:irrunitG} and \cite{bowa}, I, Thm.\ 5.3, the unique, irreducible, admissible, essentially tempered representation of $G$, which has non-trivial $(\g,K^\circ)$-cohomology twisted by $E_\mu$ is
$$A_{\q_{\underline n}}(\lambda)\otimes|\textrm{det}|^{-w/2}\cong A_{\q_{\underline n}}(\lambda)\otimes\sgn|\textrm{det}|^{-w/2},$$
where $\underline n=[0,2,2,...,2]$. As shown in the proof of Thm.\ \ref{thm:irrunitJL}, this latter representation is isomorphic to the unique irreducible quotient of
$$\textrm{Ind}^{G}_{P}[D(w,\ell_1)\otimes...\otimes D(w,\ell_k)],$$
which turns out to be irreducible itself, too, cf. \cite{hch}, Thm.\ 1, p. 198. Hence,
$$J(w,\ell):=A_{\q_{\underline n}}(\lambda)\otimes|\textrm{det}|^{-w/2}\cong \textrm{Ind}^{G}_{P}[D(w,\ell_1)\otimes...\otimes D(w,\ell_k)]$$
is the unique, irreducible, admissible, essentially tempered representation of $G$, which has non-trivial $(\g,K^\circ)$-cohomology twisted by $E_\mu$.
We compute,
$$H^q(\mathfrak{gl}_n(\R),SO(n)\R_+,J(w,\ell)\otimes E_\mu)\cong\C^2\otimes\bigwedge^{q-k^2}\C^{k-1} $$
which follows again from \cite{bowa} III. Thm 5.1 or directly using \cite{bowa}, III. Thm.\ 3.3. Observe that there is a factor $\C^2$ appearing in the formula. This is due to the fact that we calculated $(\mathfrak{gl}_n(\R),SO(n)\R_+)$-cohomology (as in Mahnkopf \cite{mahnk}, 3.1.2), and not $(\mathfrak{gl}_n(\R),O(n)\R_+)$-cohomology (as it was done by Clozel in \cite{clozel}, Lemme 3.14). The reason for that will become clear after Thm.\ \ref{thm:regalg=cusp}. Compare this also to \cite{mahnk}, p. 590.

\subsubsection*{Remaining cases}
Although they do not really come under the purview of the local Jacquet-Langlands map, let us also treat the remaining cases for further reference. If $G=G_v=GL_n(\R)$, $n=2k+1$ is odd, then, given $E_\mu$, there are two irreducible, admissible, essentially tempered representation of $G$, which have non-trivial $(\g,K^\circ)$-cohomology twisted by $E_\mu$. They are of the form
$$J(w,\ell,\varepsilon):=\textrm{Ind}^{G}_{P}[D(w,\ell_1)\otimes...\otimes D(w,\ell_k)\otimes \sgn^\varepsilon|\textrm{det}|^{-w/2}],$$
where $P$ is the standard parabolic subgroup of $G$ of type $(2,...,2,1)$ and $\varepsilon\in\{0,1\}$. Among the two representations $J(w,\ell,0)$ and $J(w,\ell,1)$, there is precisely one representation which has non-trivial $(\g,K)$-cohomology with respect to $E_\mu$. See Mahnkopf \cite{mahnk}, 3.1.3. One computes
$$H^q(\mathfrak{gl}_n(\R),SO(n)\R_+,J(w,\ell,\varepsilon)\otimes E_\mu)\cong \bigwedge^{q-k(k+1)}\C^{k}.$$

Finally, if $v$ is complex and $G=G_v=GL_n(\C)$, then, given $E_\mu$, there is again only one irreducible, admissible, essentially tempered representation $J(\mu)$ of $G$, which has non-trivial $(\g,K^\circ)$-cohomology twisted by $E_\mu$. It is again fully induced. For $\mu=(\mu_\iota,\mu_{\bar\iota})$, we define $a_i:=\mu_{\bar\iota,i}+\rho_{\iota,i}=\mu_{\bar\iota,i}+\tfrac{n-2i+1}{2}$, $1\leq i\leq n$, and set $w$ as in Sect.\ \ref{sect:highweights}. Then,
$$J(\mu)\cong\textrm{Ind}^{G}_{B}[z^{a_1-w}_1\bar z^{-a_1}_1\otimes...\otimes z^{a_n-w}_n\bar z^{-a_n}_n],$$
where $B$ is the standard Borel subgroup of $G$ and $z^x_i\bar z^y_i$ is the character of the $i$-th $GL_1(\C)$-factor of $T$, which maps $z_i$ to $z^x_i\bar z^y_i$. This follows from Enright \cite{enright}, Thm. 6.1. We obtain

$$H^q(\mathfrak{gl}_n(\C),U(n)\C^*,J(\mu)\otimes E_\mu)\cong\bigwedge^{q-n(n-1)/2}\C^{n-1}.$$

\begin{rem}\label{rem:independent}
We want to point out, that in all cases the cohomology is independent of the given highest weight representation. Cf.\ Thm.\ \ref{prop:irrunitG'}.
\end{rem}

\subsection{The global Jacquet-Langlands map}\label{sect:globalJL}
There is also a global version of the Jacquet-Langlands map, developed by Badulescu in \cite{ioan}, sect. 5, and by Badulescu--Renard in \cite{ioanrenard}, sect. 18. (For $m=1$ see also Harris--Taylor \cite{harris}, VI.1) We denote it here by $JL$ (because it goes in the direction different to the one of $|LJ|_v$). It is uniquely determined as being the map satisfying the conditions of the following theorem:

\begin{thm}[
\cite{ioanrenard}, Thm.\ 18.1 \& Prop.\ 18.2]\label{thm:globJL}
The map $JL$ is the unique injection from the set of all unitary discrete series representations $\tilde\Pi'$ of $G'(\A)$ into the set of all unitary discrete series representations $\tilde\Pi$ of $G(\A)$ such that
$$|LJ|_v(JL(\tilde\Pi')_v)=\tilde\Pi'_v$$
at all places $v\in V$. \\Moreover, if $JL(\tilde\Pi')$ is unitary cuspidal, then so is $\tilde\Pi'$.
\end{thm}

We remark that the condition $|LJ|_v(JL(\tilde\Pi')_v)=\tilde\Pi'_v$ implies that in particular $JL(\tilde\Pi')_v=\tilde\Pi'_v$ at all split places $v$. In accordance with our local definition (cf. Section \ref{sect:localJL}), we define for $\Pi'=\tilde\Pi'\otimes |\textrm{det}'|^s\in\mathscr D(G')$:
$$JL(\Pi'):=JL(\tilde\Pi')\otimes |\det|^s.$$

\begin{cor}\label{cor:globalJL}
The so extended global Jacquet-Langlands map satisfies $JL(\Pi')\in\mathscr D(G)$. Furthermore, all assertions of Thm.\ \ref{thm:globJL} hold unchanged (omitting the word ``unitary'').
\end{cor}

\section{Algebraic and regular algebraic representations of $G'(\A)$}\label{sect:regalg}
\subsection{The Local Langlands Classification}\label{sect:LLC}
Let $W_\mathbb K$ be the local Weil group of $\mathbb K=\R, \C$, cf. Tate \cite{tate}, 1.4.3. It can be defined as follows:

$$W_{\mathbb K}=\left\{
\begin{array}{ll}
\C^* & \textrm{ if $\mathbb K=\C$}\\
\C^*\cup j\C^* & \textrm{ if $\mathbb K =\R$},
\end{array}
\right.$$
where in the second case $j^2=-1$ and $jzj^{-1}=\overline z$, for all $z\in\C^*$.

It was proved by Langlands in \cite{lang2} that there is a canonical bijection between the class of irreducible admissible representations $\pi$ of $GL_n(\K)$ and the class of semi-simple, complex, $n$-dimensional representations $\tau$ of $W_\K$, $\K=\R,\C$. We denote this correspondence by
$$\pi\leftrightarrow \tau=\tau(\pi)$$
and call $\tau(\pi)$ the {\it Langlands parameter} of $\pi$. If $\K=\R$ we can furthermore restrict $\tau(\pi)$ to the connected component of the local Weil group $W_\R^\circ\cong\C^*$, which gives us an $n$-dimensional, complex representation of $\C^*$, again denoted by $\tau(\pi)$. This enables us to view $\tau(\pi)$ -- no matter if $\K=\R$ or $\K=\C$ -- as a representation $$\tau(\pi):\C^*\ra GL_n(\C).$$
If $\K=\R$ and $\pi\cong D(l)\otimes |\det|^s$ (notation as in Section \ref{sect:easylemma}) for some integer $l\geq 1$ and $s\in\C$, then the restriction of $\tau(\pi)$ to $\C^*$ is explicitly given by
$$z\mapsto (z^{\frac l2+s}(\overline z)^{-\frac l2+s})\oplus (z^{-\frac l2+s}(\overline z)^{\frac l2+s}).$$

\subsection{The split case}
We will now recall the definition of an algebraic and a regular algebraic automorphic representation of $GL_n/F$, following Clozel.

\begin{defn}[\cite{clozel}, Def. 1.8 \& Def. 3.12]\label{def:split}
Let $\Pi\in\mathscr D(G)$ with archimedean component $\Pi_\infty=\bigotimes_{v\in V_\infty}\Pi_v$.
\begin{enumerate}
\item The representation $\Pi$ is called {\it algebraic}, if at all places $v\in V_\infty$
$$\Pi_v\leftrightarrow\tau(\Pi_v)=\chi_{1,v}\oplus...\oplus\chi_{n,v},$$
where $\chi_{i,v}|\cdot|_\C^{\frac{1-n}{2}}$, $1\leq i\leq n$, is an algebraic character and $|\cdot|_\C$ denotes the usual normalized absolute value on $\C$.
\item An algebraic representation $\Pi$ is furthermore called {\it regular algebraic}, if the infinitesimal character $\chi_{\Pi_\infty}$ of $\Pi_\infty$ is regular, i.e., inside the positive, open Weyl chamber.
\end{enumerate}
\end{defn}

\begin{rem}\label{rem:algregalg}
A character $\chi_{i,v}|\cdot|_\C^{\frac{1-n}{2}}$ being algebraic means that
$$\chi_{i,v}(z)=z^{p_{i,v}}(\overline z)^{q_{i,v}}$$
with $p_{i,v},q_{i,v}\in\frac{n-1}{2}+\Z$ for all $v\in V_\infty$ and $1\leq i\leq n$. If $v$ is a real place, we can even more suppose that $\tau(\Pi_v)$ is given by
$$z\mapsto (z^{p_{1,v}}(\overline z)^{p_{2,v}})\oplus (z^{p_{2,v}}(\overline z)^{p_{1,v}})\oplus ...\oplus ((z\overline z)^{p_{2r+1,v}})\oplus...\oplus ((z\overline z)^{p_{n,v}}).$$
We let $p_v=(p_{1,v},...,p_{n,v})\in(\frac{n-1}{2}+\Z)^n$ for all $v\in V_\infty$. The definition of $\Pi$ being regular algebraic can now be equivalently reformulated by saying that for all $v\in V_\infty$, $p_{i,v}\neq p_{j,v}$, $1\leq i<j\leq n$. (This is actually the original definition of Clozel, cf. \cite{clozel}, Def. 3.12.)
\end{rem}

There is also the following, very useful, equivalent description of regular algebraic, cuspidal representations $\Pi$, which is a consequence of Clozel \cite{clozel}, Lemme 3.14.
\begin{thm}\label{thm:regalg=cusp}
Let $\Pi\in\mathscr D(G)$ and assume $\Pi$ is cuspidal. Then the following are equivalent:
\begin{enumerate}
\item[(i)] $\Pi$ is regular algebraic.
\item[(ii)] $\Pi_\infty$ is cohomological and essentially tempered.
\end{enumerate}
\end{thm}
\begin{proof}[Sketch of proof]
By Clozel \cite{clozel}, Lemme 3.14, a cuspidal automorphic representation $\Pi$ is regular algebraic if and only if $\Pi_\infty$ is cohomological. So, we only need to show that the archimedean component $\Pi_\infty$ of a cuspidal {\it and} cohomological representation $\Pi$ is automatically essentially tempered. This is actually well-known, so we sketch the argument here. For $n=1$, this is obvious. If $n\geq 2$, then by Shalika \cite{shal}, corollary on p. 190, $\Pi_v$ is a generic representation of $G_v$ for all $v\in V$. But a generic, cohomological, irreducible admissible representation of $G_v$, $v\in V_\infty$, is essentially tempered. This shows the claim.
\end{proof}

\begin{rem}
We want to point out that it is due to the fact that we calculate $(\gl_n(\R),SO(n)\R_+)$ - cohomology at a real place $v$ that we can formulate the theorem in that way. If we calculated $(\gl_n(\R),O(n)\R_+)$-cohomology instead, if $n$ is odd, one might have to twist $\Pi_v$ with the sign-character in order to get non-vanishing cohomology, as only one of the representations $J(w_v,\ell_v,0)$ and $J(w_v,\ell_v,1)$, cf.\ Sect.\ \ref{sect:tempcoh}, is $(\gl_n(\R),O(n)\R_+)$-cohomological. See also \cite{clozel}, Lem.\ 3.14.
\end{rem}

\begin{rem}\label{rem:regalg=cusp}
By Thm.\ \ref{thm:regalg=cusp} and the results of Section \ref{sect:tempcoh}, a cuspidal automorphic representation $\Pi\in\mathscr D(G)$ is regular algebraic if and only if at all archimedean places $v\in V_\infty$, the local component $\Pi_v$ is one of the according representations introduced in Section \ref{sect:tempcoh}. I.e., for a real place $v$, $\Pi_v\cong J(w_v,\ell_v)$ (resp.\ $\Pi_v\cong J(w_v,\ell_v,\epsilon_v)$) if $n$ is even (resp.\ if $n$ is odd) and $\Pi_v\cong J(\mu_v)$ for a complex place.
\end{rem}

Moreover, we remark that the cuspidality assumption on $\Pi$ cannot be removed in Thm.\ \ref{thm:regalg=cusp}. To see this, consider the following counter-example:

\begin{ex}\label{ex:(ii)not(i)}
Let $F=\Q$ and take $\pi$ to be a cuspidal automorphic representation of $GL_2(\A)$ with $\pi_\infty\cong D(2)$ (notation as in Section \ref{sect:easylemma}). 
If we put moreover $n=4$, then the global induced representation
$$\textrm{Ind}^{GL_4(\A)}_{GL_2(\A)\times GL_2(\A)}[\pi|\det|^{1/2}\otimes\pi|\det|^{-1/2}]$$
has a unique irreducible quotient $\Pi$, which is a residual representation in $\mathscr D(GL_4)$, see M\oe glin--Waldspurger \cite{moewalgln}. At infinity, $\Pi_\infty$ is the Langlands quotient of
$$\textrm{Ind}^{GL_4(\R)}_{GL_2(\R)\times GL_2(\R)}[D(2)|\det|^{1/2}\otimes D(2)|\det|^{-1/2}].$$
Therefore, by Section \ref{sect:LLC}, $\tau(\Pi_\infty)$ is given by the character
$$z\mapsto (z^{\frac{3}{2}}(\overline z)^{-\frac{1}{2}})\oplus (z^{-\frac{1}{2}}(\overline z)^{\frac{3}{2}})\oplus (z^{\frac{1}{2}}(\overline z)^{-\frac{3}{2}})\oplus (z^{-\frac{3}{2}}(\overline z)^{\frac{1}{2}}),$$
whose exponents are all in $\frac12+\Z$ and pairwise different, whence $\Pi$ is regular algebraic by Rem. \ref{rem:algregalg}. Its archimedean component $\Pi_\infty$ is even cohomological with respect to the trivial representation: One may check using Thm.\ \ref{thm:irrunitJL}, resp. Table \ref{t:g}, that
$$H^q(\gl_4(\R),SO(4)\R^+,\Pi_\infty\otimes\C)\neq 0$$
in degrees $q=3,6$. But clearly, $\Pi_\infty$ is not tempered. Indeed, $\Pi_\infty = A_{\q_{\underline n}}(\lambda)$ with $\lambda=0$ and $\underline n=[0,4]$, which is not tempered by Thm.\ \ref{prop:irrunitG} (3).
\end{ex}

Before we turn to the case of $G'$, let us also recall Clozel's ``Lemme de puret\'e'', \cite{clozel}, Lemme 4.9:

\begin{lem}\label{lem:pureteG}
If $\Pi\in\mathscr D(G)$ is algebraic and cuspidal, then there is a {\sf w} $\in\Z$ such that for all $v\in V_\infty$, the algebraic characters of $\C^*$ associated to $\Pi_v|\cdot|_v^{\frac{1-n}{2}}$ are of the form $z\mapsto z^p(\overline z)^q$ with $p+q=$ {\sf w}.
\end{lem}

\medskip\medskip
\subsection{The general case}
Motivated by the above definition for the split case (i.e., $D=F$) we extend it to the more general case of $G'$.

\begin{defn}\label{def:nonsplit}
Let $\Pi'\in\mathscr D(G')$.
\begin{enumerate}
\item The representation $\Pi'$ is called {\it algebraic}, if $JL(\Pi')\in\mathscr D(G)$ is algebraic.
\item An algebraic representation $\Pi'$ is called {\it regular algebraic}, if $JL(\Pi')\in\mathscr D(G)$ is regular algebraic.
\end{enumerate}
\end{defn}

We will now see that our definition goes very well with Thm.\ \ref{thm:regalg=cusp}.

\begin{thm}\label{thm:regalgG'}
Let $\Pi'\in\mathscr D(G')$ and assume that $JL(\Pi')$ is cuspidal. Then the following are equivalent:
\begin{enumerate}
\item[(i)] $\Pi'$ is regular algebraic.
\item[(ii)] $\Pi'_\infty$ is cohomological and essentially tempered.
\end{enumerate}
\end{thm}
\begin{proof}
(i)$\Rightarrow$(ii): By assumption $JL(\Pi')$ is a cuspidal, regular algebraic representation of $G(\A)$. So, Thm.\ \ref{thm:regalg=cusp} implies that $JL(\Pi')_\infty=\bigotimes_{v \in V_\infty}JL(\Pi')_v$ is cohomological and essentially tempered. By the description of the global Jacquet-Langlands map, cf. Thm.\ \ref{thm:globJL} and Cor.\ \ref{cor:globalJL}, there is an isomorphism $JL(\Pi')_v\cong\Pi'_v$ at all split places $v\in V$. In particular, $\Pi'_v$ is proved to be cohomological and essentially tempered at all split places $v\in V_\infty$. So, let us assume from now on that $v$ is a real, non-split place. Again by Thm.\ \ref{thm:globJL} and Cor.\ \ref{cor:globalJL} we obtain
$$\Pi'_v \cong |LJ|_v(JL(\Pi')_v).$$
As $v$ is supposed to be non-split, we have $n=2k$ and so Thm.\ \ref{thm:regalg=cusp} (see also Rem. \ref{rem:regalg=cusp}) implies that
$$JL(\Pi')_v\cong J(w_v,\ell_v),$$
for some integer $w_v$ and $k$-tuple $\ell_v$ as in Section \ref{sect:tempcoh}. But now Thm.\ \ref{thm:irrunitJL} shows that $|LJ|_v(J(w_v,\ell_v))\cong J'(w_v,\ell_v)$. Putting the pieces together we finally see that
$$\Pi'_v \cong J'(w_v,\ell_v).$$
According to Thm.\ \ref{prop:irrunitG'} (resp., Section \ref{sect:tempcoh}), $\Pi'_v$ is therefore cohomological and essentially tempered and hence so is $\Pi'_\infty=\bigotimes_{v \in V_\infty}\Pi'_v$.

(ii)$\Rightarrow$(i): Assume that $\Pi'_\infty$ is cohomological and essentially tempered. By definition, cf. Def. \ref{def:nonsplit}, we need to show that $JL(\Pi')$ is regular algebraic. As this is a local condition at the archimedean places, Thm.\ \ref{thm:regalg=cusp} together with the standing assumption that $JL(\Pi')$ is cuspidal, permits us again to focus on the non-split archimedean places $v\in V_\infty$. Let $v\in V_\infty$ be such a place and let $E_{\mu_v}$, $\mu_v=(\mu_{v,1},...,\mu_{v,n})$, be a highest weight representation of $G'_v=GL_k(\H)$ with respect to which $\Pi'_v$ is cohomological. We know from Section \ref{sect:tempcoh} that $E_{\mu_v}$ is necessarily essentially self-dual, so there is an integer $w_v\in\Z$ such that $w_v=\mu_{v,i}+\mu_{v,n-i+1}$, $1\leq i\leq n$. As $\Pi'_v$ is furthermore assumed to be essentially tempered, Thm.\ \ref{prop:irrunitG'}, resp., our explanations of Section \ref{sect:tempcoh}, imply that $$\Pi'_v\cong J'(w_v,\ell_v),$$ the $k$-tuple $\ell_v=(\ell_{1,v},...,\ell_{k,v})$ being defined as in Section \ref{sect:tempcoh}. By the description of the global Jacquet-Langlands map, we hence obtain $$J'(w_v,\ell_v) \cong |LJ|_v(JL(\Pi')_v).$$ Next, recall that $JL(\Pi')_v$ -- being the local component of a cuspidal representation -- must be generic (see Shalika \cite{shal}, corollary on p. 190). Therefore $JL(\Pi')_v$ is induced from representations in the (limits of) essentially discrete series, cf. Vogan \cite{vogangen}, Thm.\ 6.2. But since $|LJ|_v$ commutes with induction and tensor products, these representations must be all in $U_{cp}(GL_2(\R))$ and so Lem.\ \ref{lem:localJL} shows that $JL(\Pi')_v$ is induced from the representations $D(w_v,\ell_{i,v})$, $1\leq i\leq k$. Whence we have by construction
$$JL(\Pi')_v\cong J(w_v,\ell_v)$$
and so Thm.\ \ref{prop:irrunitG} (resp., Section \ref{sect:tempcoh}) yields that $$JL(\Pi')_\infty=\bigotimes_{v \in V_\infty}JL(\Pi')_v$$ is cohomological and essentially tempered. By Thm.\ \ref{thm:regalg=cusp}, $JL(\Pi')$ is therefore regular algebraic. Now the proof is complete.
\end{proof}

\begin{rem}\label{ex:(ii)->(i)}
Implication (ii)$\Rightarrow$(i) does not seem to need the assumption that $JL(\Pi')$ is cuspidal. First, observe that -- assuming the validity of (ii) -- $\Pi'$ is necessarily cuspidal by Wallach \cite{wallach}, Thm.\ 4.3. See also Clozel \cite{clozel2}, Prop.\ 4.10. By the same references, $JL(\Pi')$ is cuspidal, if there is a split archimedean place $v\in V_\infty$. It therefore suffices to prove the implication (ii)$\Rightarrow$(i) under the assumption that there are only non-split places $v\in V_\infty$. It seems that $\Pi'_v$, being cohomological and essentially tempered, satisfies
$$JL(\Pi')_v\cong A_{\q_{\underline n}}(\lambda)\otimes\textrm{sgn}^\varepsilon|\textrm{det}|^{-w/2},\quad\quad\quad(\heartsuit)$$
with $\underline n=[0,2,2,...,2]$ or $\underline n=[2,2,...,2]$. If $\underline n=[0,2,2,...,2]$, then $JL(\Pi')_v$ is essentially tempered and we are again in the situation considered in Thm.\ \ref{thm:regalgG'}. If $\underline n=[2,2,...,2]$, then the local Langlands parameter of $JL(\Pi')_v$ is given by
$$(z^\frac{\ell_{1}-w}{2}(\overline z)^\frac{-\ell_1-w}{2})\oplus(z^\frac{-\ell_{1}-w}{2}(\overline z)^\frac{\ell_1-w}{2})\oplus...\oplus(z^\frac{\ell_{k-1}-w}{2}(\overline z)^\frac{-\ell_{k-1}-w}{2})\oplus(z^\frac{-\ell_{k-1}-w}{2}(\overline z)^\frac{\ell_{k-1}-w}{2})\oplus(z\overline z)^{\frac{\ell_k-w}{2}}\oplus(z\overline z)^{\frac{-\ell_k-w}{2}},$$
with $\ell=(\ell_1,...,\ell_k)$ and $w$ as in Section \ref{sect:tempcoh}; in particular, $\ell_k=1$. Indeed, for $A_{\q_{\underline n}}(\lambda)$ with $\underline n=[2,2,...,2]$ to exist, we must have $\lambda_k=0$ by \eqref{eq:extending}. As furthermore $\rho_k^{\u_{\underline n}}=0$, the representation giving $(z\overline z)^{\frac{\ell_k-w}{2}}\oplus(z\overline z)^{\frac{-\ell_k-w}{2}}$ is $|\det|^{-w/2}$, which shows $\ell_k=1$. By the definition of the numbers $\ell_i$ we see that $\ell_i-w=2(-\mu_{n-i+1}+k-i)+1$ is odd, so $JL(\Pi')_v$ would be algebraic. But since also $\ell_1>\ell_2>...>\ell_{k-1}>\ell_k= 1$, $JL(\Pi')_v$ would also be regular algebraic by Rem. \ref{rem:algregalg}. As a consequence, in order to show the implication (ii)$\Rightarrow$(i) without the assumption that $JL(\Pi')$ is cuspidal, it suffices to prove $(\heartsuit)$.
\end{rem}

In contrast, the implication (i)$\Rightarrow$(ii) in Thm.\ \ref{thm:regalgG'} fails, if one drops the assumption of the cuspidality of $JL(\Pi')$ -- even if one supposes that $\Pi'$ is cuspidal. This is shown in the next example:

\begin{ex}[Grobner \cite{grob}, Thm.\ 4.1]\label{ex:(i)not(ii)}
Let $F=\Q$, $m=2$ and $d=2$, i.e., $G'$ is the group of $GL_2(D)$ over a quaternion division algebra $D$ over $\Q$. Assume furthermore that the infinite place is non-split, so $G'_\infty=GL_2(\H)$. Using the global Jacquet-Langlands map of \cite{ioanrenard}, it was shown by the first named author in \cite{grob}, Thm.\ 4.1, that for any integer $k\geq 0$, there is a cuspidal automorphic representation $\Pi'$ of $G'(\A)$ whose component at infinity $\Pi'_\infty$ is given by the Langlands quotient of
$$\textrm{Ind}^{GL_2(\H)}_{GL_1(\H)\times GL_1(\H)}[\textrm{Sym}^{k+1}\C^2\textrm{det}'^{1/2}\otimes \textrm{Sym}^{k+1}\C^2\textrm{det}'^{-1/2}].$$
Therefore, $\Pi'_\infty$ is not essentially tempered, cf. \cite{grob}, Prop.\ 3.5. (Although, if $k$ is even, $\Pi'_\infty$ is cohomological with respect to the highest weight representation $E_\mu$ of $G'_\infty$ with $\mu=(\frac{k}{2},\frac{k}{2},-\frac{k}{2},-\frac{k}{2})$; see again \cite{grob}, Prop.\ 3.5, resp., our Thm.\ \ref{prop:irrunitG'}.)

But if $k$ is chosen to be even, we claim that $\Pi'$ is regular algebraic. In fact, the Langlands parameter $\tau(JL(\Pi')_\infty)$ is given by Section \ref{sect:LLC} as
$$z\mapsto (z^\frac{(k+3)}{2}(\overline z)^\frac{-(k+1)}{2})\oplus (z^\frac{-(k+1)}{2}(\overline z)^\frac{(k+3)}{2})\oplus (z^\frac{(k+1)}{2}(\overline z)^\frac{-(k+3)}{2})\oplus (z^\frac{-(k+3)}{2}(\overline z)^\frac{(k+1)}{2}),$$
whose exponents are in $\frac12+\Z$, if $k$ is even. So, $\Pi'$ is algebraic for even $k$. Furthermore, the numbers
$$p_{1,\infty}=\frac{(k+3)}{2}, p_{2,\infty}=\frac{-(k+1)}{2}, p_{3,\infty}=\frac{(k+1)}{2}, p_{4,\infty}=\frac{-(k+3)}{2}$$
can never be pairwise equal. Hence $\Pi'$ is regular algebraic for even $k$ by Rem. \ref{rem:algregalg}.
\end{ex}

\subsection{A purity lemma}
We conclude this section by stating a generalization of Clozel's ``Lemme de puret\'e''.

\begin{lem}\label{lem:pureteG'}
If $\Pi'\in\mathscr D(G')$ is algebraic and $JL(\Pi')$ cuspidal, then there is a {\sf w} $\in\Z$ such that for all $v\in V_\infty$, the characters of $\C^*$ associated to $JL(\Pi')_v|\cdot|_v^{\frac{1-n}{2}}$ are of the form $z\mapsto z^p(\overline z)^q$ with $p+q=$ {\sf w}.
\end{lem}
\begin{proof}
This is an immediate consequence of the definition of $\Pi'$ being algebraic and Lem.\ \ref{lem:pureteG}.
\end{proof}

Again, the cuspidality of $JL(\Pi')$ cannot be weakened to assuming that $\Pi'$ is cuspidal, as the next example shows.

\begin{ex}\label{ex:purity}
Once more consider the representation $\Pi'$ constructed in Example \ref{ex:(i)not(ii)}. It is cuspidal and algebraic, if $k$ is even. Recalling also its Langlands parameter from Example \ref{ex:(i)not(ii)}, one checks that there is no integer {\sf w} $\in\Z$ representing all the sums of the exponents showing up in $\tau(JL(\Pi')_\infty)|\cdot|_\infty^{\frac{1-n}{2}}$.
\end{ex}

\section{Spaces of automorphic cohomology and rational structures}

\subsection{Definition of the $\sigma$-twist, the rationality field and $\F$-structures}\label{sect:deftwist}
Let $\nu$ be any representation of either $G'(\A_f)$ or $G'_v$, $v\in V_f$, on a complex vector space $W$. For $\sigma\in \textrm{Aut}(\C)$, we define the {\it $\sigma$-twist} ${}^\sigma\!\nu$ following Waldspurger \cite{waldsp}, I.1: If $W'$ is a $\C$-vector space which admits a $\sigma$-linear isomorphism $t:W\ra W'$ then we set
$${}^\sigma\!\nu:= t\circ\nu\circ t^{-1}.$$
This definition is independent of $t$ and $W'$ up to equivalence of representations. One may hence always take $W':=W\otimes_\C {}_{\sigma}\C$. This latter vector space is defined as follows: Firstly, since $\C$ is commutative it makes no difference whether we think of $W$ either as a right or a left vector space. For concreteness, we will think of $W$ as a right $\C$-vector space. Secondly, consider $\C$ as a $(\C,\C)$-bimodule, where the left module structure is via an automorphism $\sigma \in$Aut$(\C)$, i.e., $a \cdot_l z := \sigma(a)z,$ and the right module structure is by usual multiplication: $z \cdot_r b = zb.$ Let us denote this bimodule simply as ${}_\sigma\C$.
Now consider the tensor product: $W \otimes_{\C} {}_\sigma\C$, which uses only the left-module structure on ${}_\sigma\C$. For all $a \in \C$, $w \in W$ and $z \in \C$, we have: $aw \otimes z = wa \otimes z = w \otimes \sigma(a)z.$ Further, using the right module structure of ${}_\sigma\C$, we get that $W \otimes_{\C} {}_\sigma\C$ is a right $\C$-vector space, where the scalar multiplication is given by:
$$
a(w \otimes z) = w \otimes za.
$$
It is now a direct consequence of this definition that the canonical map $t : W \to W \otimes_{\C} {}_\sigma\C$ defined by $t(w) = w \otimes 1$
is a $\sigma$-linear isomorphism:
$t(aw) = aw \otimes 1 = w \otimes \sigma(a) = \sigma(a)(w \otimes 1) = \sigma(a)t(w).$\\

On the other hand, let $\nu=E_\mu$ be a highest weight representation of $G'_\infty$ as in Sect.\ \ref{sect:highweights}. The group Aut$(\C)$ acts on the set of embeddings $\Sigma_\infty=\{\iota:F\hra\C\}$ by composition. Hence, we may define ${}^\sigma\!E_\mu$ to be the irreducible representation of $G'_\infty$, whose local factor at the embedding $\iota$ is $E_{\mu_{\sigma^{-1}\iota}}$, i.e., has highest weight $\mu_{\sigma^{-1}\iota}$. This way, $\sigma$ may very well mix up real and complex places. Suppose now that $E_\mu$ is regular. As this is a purely local condition on the factors at the embeddings, so is ${}^\sigma\!E_\mu$.\\

Recall also the definition of the rationality field of a representation from \cite{waldsp}, I.1. If $\nu$ is any of the representations considered above, then let $\mathfrak S(\nu)$ be the group of all automorphisms $\sigma\in \textrm{Aut}(\C)$ such that ${}^\sigma\!\nu\cong\nu $:
$$\mathfrak S(\nu):=\{\sigma\in \textrm{Aut}(\C)|{}^\sigma\!\nu\cong\nu\}.$$
Then the {\it rationality field} $\Q(\nu)$ is defined as
$$\Q(\nu):=\{z\in\C| \sigma(z)=z \textrm{ for all } \sigma\in\mathfrak S(\nu)\}.$$
As a third ingredient we recall that a group representation $\nu$ on a $\C$-vector space $W$ is said to be {\it defined over a subfield $\F\subset\C$}, if there is a $\F$-vector subspace $W_\F\subset W$, stable under the group action, and such that the canonical map $W_\F\otimes_\F\C\rightarrow W$ is an isomorphism. In this case, we also say that $(\nu,W)$ has an {\it $\F$-structure}.

In particular, if $E_\mu=\bigotimes_{v\in V_\infty} E_{\mu_v}$ is a highest weight representation of $G'_\infty$, then the above definitions can be applied to $E_\mu$. So, let $\Q(E_\mu)$ be the rationality field of $E_\mu$ and let $L/F$ be a minimal algebraic field extension of $F$, such that $D$ splits over
$$\Q(\mu):=\Q(E_\mu)\cdot L,$$
the compositum of the fields $\Q(E_\mu)$ and $L$, i.e., $D\otimes_F\Q(\mu)\cong M_d(\Q(\mu))$.

\begin{lem}\label{lem:Erational}
Let $E_\mu$ be an irreducible highest weight representation of $G'_\infty$ as in Section \ref{sect:highweights}. Let $\sigma\in$ \emph{Aut}$(\C)$ and consider the $\sigma$-twisted representation ${}^\sigma\!E_\mu$ of $G'_\infty$. As an abstract representation of the diagonally embedded group $G'(F)\hra G'_\infty$, ${}^\sigma\!E_\mu$ is isomorphic to $E_\mu\otimes_{\C} {}_{\sigma}\C$. Furthermore, as a representation of $G'(F)$, $E_\mu$ is defined over $\Q(\mu)$, which is in turn a number field. 
\end{lem}
\begin{proof}
Let $\psi$ be the isomorphism of $\C$-vector spaces $\psi:{}^\sigma\!E_\mu\ra E_\mu\otimes_{\C} {}_{\sigma}\C$, given by the assignment
$\psi(\otimes_{\iota\in \Sigma_\infty} w_\iota):=\otimes_{\iota\in \Sigma_\infty} \sigma^{-1}(w_{\sigma \iota})$. Applying $\sigma^{-1}$ to a vector $w_{\sigma \iota}\in E_{\mu_{\iota}}$ is meant as applying $\sigma^{-1}$ to the coordinates of $w_{\sigma \iota}$ in terms of the standard basis of the standard representation $\C^n$: This is well-defined since at every real archimedean place, $E_{\mu_v}=E_{\mu_\iota}$ is a subrepresentation of
$$\nu_{\mu_v}:=\bigotimes_{i=1}^{n-1}\textrm{Sym}^{\mu_{v,i}-\mu_{v,i+1}}\left(\Lambda^i\C^n\right)\otimes \textrm{det}'^{\mu_{v,n}},$$
while at every complex place, $E_{\mu_v}$ is a subrepresentation of $\nu_{\mu_{\iota_v}}\otimes\bar\nu_{\mu_{\bar\iota_v}}$. Recall that $g\in G'(F)$, being diagonally embedded into $G'_\infty$, acts on $w\in {}^\sigma\! E_\mu$ as $g\cdot w=\otimes_{\iota\in \Sigma_\infty} \iota(g)w_\iota$. So, invoking the Skolem-N\oe ther-Theorem,
$$\psi(g\cdot w)= \otimes_{\iota\in \Sigma_\infty} \sigma^{-1}((\sigma\circ \iota)(g) w_{\sigma \iota})=\otimes_{\iota\in \Sigma_\infty} \iota(g)\sigma^{-1}(w_{\sigma \iota})=g\cdot\psi(w),$$
which shows that $\psi$ induces an isomorphism of the $G'(F)$-modules ${}^\sigma\!E_\mu$ and $E_\mu\otimes_{\C} {}_{\sigma}\C$.

In order to prove the second assertion, let us view $E_\mu$ as a representation of $G(F)=GL_n(F)$ via its diagonal embedding into $G_\infty$. Exactly as in Clozel \cite{clozel}, p.122, one sees that $E_\mu$ is defined as a representation of $G(F)$ over $\Q(E_\mu)$, which in turn is also a number field. Let $E_{\mu,\Q(E_\mu)}$ be a $\Q(E_\mu)$-structure of $E_\mu$ as a representation of $G(F)$. Since the algebraic group $G\times \Q(\mu)=GL_n/\Q(\mu)$ is connected and char$(\Q(\mu))=0$, the $\Q(\mu)$-subspace $E_{\mu,\Q(\mu)}:=E_{\mu,\Q(E_\mu)}\otimes_{\Q(E_\mu)}\Q(\mu)$ of $E_\mu$ is stable by the action of $G(\Q(\mu))$. Hence, it is stable by the action of $G'(F)\subseteq G'(\Q(\mu))=G(\Q(\mu))$. Therefore, $E_\mu$ is defined as a representation of $G'(F)$ over the number field $\Q(\mu)$. This proves the second assertion.

\end{proof}

\subsection{Rational structures on the cohomology of geometric spaces}\label{sect:geomcoh}
Consider now the quotient
$$S_{G'}:= G'(F)\backslash G'(\A)/K'^\circ_\infty,$$
where we recall that the group $\R_+$ is ``hidden'' as a diagonal subgroup of $K'^\circ_\infty$. This space is the projective limit of finite disjoint unions of orbifolds, cf.\ Rohlfs \cite{roh}, Prop.\ 1.9 together with Borel \cite{bor1}, Thm.\ 5.1. Let $E_\mu$ be a highest weight representation of $G'_\infty$. It defines a sheaf $\mathcal E_\mu$ on $S_{G'}$, by letting $\mathcal E_\mu$ be the sheaf with espace \'etal\'e $G'(\A)/K'^\circ_\infty \times_{G'(F)} E_\mu$ with the discrete topology on $E_\mu$. Hence, the sheaf cohomology spaces
$$H^q(S_{G'},\mathcal E_\mu)$$
are defined. These are $G'(\A_f)$-modules. Moreover, the sheaf-cohomology with compact support $H_c^q(S_{G'},\mathcal E_\mu)$ is well-defined, too, and there is the natural map
$$H_c^q(S_{G'},\mathcal E_\mu)\rightarrow H^q(S_{G'},\mathcal E_\mu).$$
We denote its image, following Harder, by $H_!^q(S_{G'},\mathcal E_\mu)$ and call it {\it interior cohomology}.

It will be important for us to interpret interior cohomology as the kernel of another map. Therefore, recall the adelic Borel-Serre Compactification $\overline S_{G'}$ of
$S_{G'}$, and its basic properties. (For this we refer to Borel--Serre \cite{boserre} as the original source and to Rohlfs \cite{roh} for the adelic setting.) It is a compact space with boundary $\partial(\overline S_{G'})$ and the inclusion
$S_{G'}\hra\overline S_{G'}$ is a homotopy equivalence. Furthermore, the sheaf $\mathcal E_\mu$ extends to a sheaf on $\overline S_{G'}$ and there is the natural restriction morphism of $G'(\A_f)$-modules
$$res^q: H^q(S_{G'},\mathcal E_\mu)\cong H^q(\overline S_{G'},\mathcal E_\mu)\ra H^q(\partial(\overline S_{G'}),\mathcal E_\mu).$$
It is a basic fact that
\begin{equation}\label{eq:intcoh}
H_!^q(S_{G'},\mathcal E_\mu) = \ker (res^q),
\end{equation}
which follows from considering the long exact sequence in sheaf-cohomology given by the pair $(\overline S_{G'},\partial(\overline S_{G'}))$. We obtain

\begin{lem}\label{lem:Qstructures}
The $G'(\A_f)$-modules $H^q(S_{G'},\mathcal E_\mu)$ and $H_!^q(S_{G'},\mathcal E_\mu)$ are defined over $\Q(\mu)$.
\end{lem}
\begin{proof}
Both assertions follow from the fact that sheaf-cohomology can be computed using Betti-cohomology. To be more precise, we recall from Lem.\ \ref{lem:Erational} that $E_\mu$ is defined over the field $\Q(\mu)$ as a representation of $G'(F)$. Let $E_{\mu, \Q(\mu)}$ be a $\Q(\mu)$-structure on $E_\mu$ and let $H^q_B(S_{G'},E_{\mu, \Q(\mu)})$ denote the Betti-cohomology of $S_{G'}$ with coefficients in $E_{\mu, \Q(\mu)}$. This is a $G'(\A_f)$-stable $\Q(\mu)$-vector space and defines a $\Q(\mu)$-structure:
$$H^q(S_{G'},\mathcal E_\mu)\cong H^q_B(S_{G'},E_{\mu, \Q(\mu)})\otimes_{\Q(\mu)}\C.$$
Recalling \eqref{eq:intcoh}, this also shows the analogous assertion for $H_!^q(S_{G'},\mathcal E_\mu)$.
\end{proof}
Compare the previous lemma also with Clozel \cite{clozel}, p.122-123.\\
As another ingredient observe that for all $\sigma\in \textrm{Aut}(\C)$, there are natural $\sigma$-linear, $G'(\A_f)$-equivariant isomorphisms
\begin{equation}\label{eq:sigmaiso}
\sigma^*: H^q(S_{G'},\mathcal E_\mu) \rightarrow H^q(S_{G'},{}^\sigma\!\mathcal E_\mu)
\end{equation}
and
\begin{equation}\label{eq:sigmaiso!}
\sigma_!^*: H^q_!(S_{G'},\mathcal E_\mu) \rightarrow H^q_!(S_{G'},{}^\sigma\!\mathcal E_\mu).
\end{equation}

\subsection{Combining automorphic forms and geometry}\label{sect:autcoh1}
Let $E_\mu$ be a highest weight representation of $G'_\infty$ and let $\mathcal J=\mathcal{J}_{E_\mu}$ be the ideal of the center of the universal enveloping algebra of $\g_\C'$, which annihilates the contragredient representation of $E_\mu$. This ideal is of finite codimension and hence the subspace $\mathcal A_\mathcal J(G')$ of automorphic forms on $\R_+ G'(F)\backslash G'(\A)$, which are annihilated by some power of $\mathcal J$ is defined, cf. Section \ref{sect:automf}. Suppose furthermore that $E_\mu|_{\R_+}= 1$ in this subsection. This is not the most general setup (instead one should be working with $G'(\A)^{(1)}$, cf.\ Clozel \cite{clozel} p.\ 123), but merely provides a convenient normalization. Then the link between sheaf cohomology and automorphic forms is provided by

\begin{thm}[Franke \cite{franke}, Thm.\ 18]\label{thm:franke}
There is an isomorphism of $G'(\A_f)$-modules
$$H^q(S_{G'},\mathcal E_\mu)\cong H^q(\g_\infty',K'^\circ_\infty,\mathcal A_\mathcal J(G')\otimes E_\mu).$$
\end{thm}

Abbreviating

$$H_{\{P'\}}^q(G',E_\mu):=H^q(\g_\infty',K'^\circ_\infty,\mathcal A_{\mathcal J,\{P'\}}(G')\otimes E_\mu)$$
and
$$H_{\{P'\},\varphi_{P'}}^q(G',E_\mu):=H^q(\g_\infty',K'^\circ_\infty,\mathcal A_{\mathcal J,\{P'\},\varphi_{P'}}(G')\otimes E_\mu),$$
we obtain

\begin{cor}\label{cor:cohdecP'phi}
The sheaf cohomology of $S_{G'}$ inherits from \eqref{eq:autdecP'} and \eqref{eq:autdecP'phi} a decomposition as $G'(\A_f)$-module:
\begin{eqnarray}\label{eq:cohdecP'phi}
H^q(S_{G'},\mathcal E_\mu) &\cong  & \bigoplus_{\{P'\}} H_{\{P'\}}^q(G',E_\mu)\\
 &\cong  & \bigoplus_{\{P'\}}\bigoplus_{\varphi_{P'}} H_{\{P'\},\varphi_{P'}}^q(G',E_\mu).
\end{eqnarray}
\end{cor}

\subsection{}\label{sect:autcoh2}
Recall from Section \ref{sect:automf} that the space $\mathcal A_{cusp, \mathcal J}(G')$, i.e., the subspace of all cuspidal automorphic forms in $\mathcal A_{\mathcal J}(G')$, can be identified with the summand $\mathcal A_{\mathcal J,\{G'\}}(G')$ in \eqref{eq:autdecP'}. Hence, we call the summand in \eqref{eq:cohdecP'phi} index by $\{G'\}$ the space of {\it cuspidal cohomology} and denote it
\begin{eqnarray*}
H^q_{cusp}(G',E_\mu) & := & H_{\{G'\}}^q(G',E_\mu)\\
& = & H^q(\g_\infty',K'^\circ_\infty,\mathcal A_{cusp, \mathcal J}(G')\otimes E_\mu).
\end{eqnarray*}
Now, let $\mathcal A_{dis, \mathcal J}(G')$ be the subspace of all square-integrable automorphic forms in $\mathcal A_{\mathcal J}(G')$.
The quotient space
$\R_+ G'(F)\backslash G'(\A)$ has finite volume hence $\mathcal A_{dis, \mathcal J}(G')$
is the space of all $K'^\circ_\infty$-finite, smooth functions $f\in L^2_{dis}(\R_+ G'(F)\backslash G'(\A))$, which are annihilated by some power of $\mathcal J$. We will denote its cohomology by
$$H^q_{dis}(G',E_\mu):=H^q(\g_\infty',K'^\circ_\infty,\mathcal A_{dis, \mathcal J}(G')\otimes E_\mu).$$
Observe that by \eqref{eq:cohdecP'phi}, there is a natural inclusion of $G'(\A_f)$-modules
$$H^q_{cusp}(G',E_\mu)\hookrightarrow H^q_{dis}(G',E_\mu).$$
Let us now refine this picture even more. If we fix a smooth character
$$\omega:\R_+ Z'(F)\backslash Z'(\A)\rightarrow\C^*,$$
then the subspace of functions in $\mathcal A_{dis, \mathcal J}(G')$ (resp., $\mathcal A_{cusp, \mathcal J}(G')$), which have this given central character $\omega$, decomposes as a direct Hilbert sum, the sum ranging over all (equivalence classes) of square-integrable (resp., cuspidal), irreducible automorphic representations with central character $\omega$, which are annihilated by some power of $\mathcal J$:
\begin{equation}\label{eq:disdecmp}
\mathcal A_{dis, \mathcal J}(G',\omega)\cong\widehat{\bigoplus}_{\Pi'\textrm{ square-int.}}\Pi',
\end{equation}
resp.,
\begin{equation}\label{eq:cuspdecmp}
\mathcal A_{cusp, \mathcal J}(G',\omega)\cong\widehat{\bigoplus}_{\substack{ \Pi'\textrm{ cuspidal.}}}\Pi'.
\end{equation}
Once more, these spaces define in a natural way $(\g_\infty',K'^\circ_\infty)$-modules and if $E_\mu$ is a highest weight representation of $G'_\infty$, we define
\begin{eqnarray}\label{eq:discohdec}
H^q_{dis,\omega}(G',E_\mu) & := & H^q(\g_\infty',K'^\circ_\infty,\mathcal A_{dis, \mathcal J}(G',\omega)\otimes E_\mu)\\
 & \cong & \bigoplus_{\substack{ \Pi'\textrm{ square-int.}}} H^q(\g_\infty',K'^\circ_\infty,\Pi'_\infty\otimes E_\mu)\otimes\Pi'_f
\end{eqnarray}
and
\begin{eqnarray}
H^q_{cusp,\omega}(G',E_\mu) & := & H^q(\g_\infty',K'^\circ_\infty,\mathcal A_{cusp, \mathcal J}(G',\omega)\otimes E_\mu)\\
& \cong & \bigoplus_{\substack{ \Pi'\textrm{ cuspidal}}} H^q(\g_\infty',K'^\circ_\infty,\Pi'_\infty\otimes E_\mu)\otimes\Pi'_f.\label{eq:cuspcohdec}
\end{eqnarray}
Both spaces are $G'(\A_f)$-modules and their decompositions are inherited by \eqref{eq:disdecmp}, resp., \eqref{eq:cuspdecmp}.

By Thm.\ \ref{thm:franke} it is justified to talk about the image of all cohomology spaces constructed in Section \ref{sect:autcoh1} and \ref{sect:autcoh2} in $H^q(S_{G'},\mathcal E_\mu)$. Let us denote these various images by overlining ``$H$''. In particular, this applies to $H^q_{cusp}(G',E_\mu)$ and $H^q_{dis}(G',E_\mu)$ (resp., $H^q_{cusp,\omega}(G',E_\mu)$ and $H^q_{dis,\omega}(G',E_\mu)$), so, e.g., Im$H^q_{cusp}(G',E_\mu)=:\bar H^q_{cusp}(G',E_\mu)$. Then Cor.\ \ref{cor:cohdecP'phi} implies the following proposition, which describes the interplay between interior cohomology and the various cohomology spaces constructed in this section.

\begin{prop}\label{prop:inclusions}
There is the following commutative diagram of $G'(\A_f)$-modules.
\begin{displaymath}
\xymatrix{
H^q_{cusp}(G',E_\mu) \ar[r]^\cong & \bar H^q_{cusp}(G',E_\mu) \ar@{^{(}->}[r]  &  H_!^q(S_{G'},\mathcal E_\mu) \ar@{^{(}->}[r] & \bar H^q_{dis}(G',E_\mu)\\
H^q_{cusp,\omega}(G',E_\mu) \ar[r]^\cong & \bar H^q_{cusp,\omega}(G',E_\mu) \ar@{^{(}->}[u]  \ar@{^{(}->}[rr]&   & \bar H^q_{dis,\omega}(G',E_\mu)  \ar@{^{(}->}[u]                    }
\end{displaymath}
\end{prop}

\subsection{Three results in the split case}
Having made some preparatory work in Sections \ref{sect:deftwist}--\ref{sect:autcoh2}, let us now recall one of the main results in Section 3 of Clozel \cite{clozel}:

\begin{thm}[\cite{clozel}, Thm.\ 3.13]\label{thm:3.13}
Let $\Pi\in\mathscr D(G)$ be regular algebraic and cuspidal. Then for all $\sigma\in \textrm{\emph{Aut}}(\C)$, there is a uniquely determined representation $\Xi\in\mathscr D(G)$, which is regular algebraic and cuspidal and such that $\Xi_f\cong{}^\sigma\Pi_f$.
\end{thm}
\begin{rem}
The representation $\Xi$ is uniquely determined by Strong Multiplicity One. We may hence also write ${}^\sigma\Pi:=\Xi$, abusing notation. Although it is regular algebraic, whence its archimedean component is cohomological and essentially tempered (cf.\ Thm.\ \ref{thm:regalg=cusp}), ${}^\sigma\Pi_\infty$  may not be directly determined by applying a simple ``permutation action'' to its factors: Indeed, if $n$ is odd, then there are two possible choices of a cohomological representation at a real place. See Sect.\ \ref{sect:tempcoh}. However, in any other case, i.e., if either $n$ is even, or if $F$ has no real place, ${}^\sigma\Pi_\infty$ equals the unique irreducible, admissible, essentially tempered representation which is cohomological with respect to ${}^\sigma\!E_\mu$.
\end{rem}
We may reformulate Thm.\ \ref{thm:3.13} within the setup developed in Sections \ref{sect:geomcoh}-\ref{sect:autcoh2} applied to the special case $D=F$, i.e., $G'=G$. As $\Pi$ is regular algebraic and cuspidal, $\Pi_\infty$ is cohomological with respect to some highest weight representation $E_\mu$ of $G_\infty$, cf. Thm.\ \ref{thm:regalg=cusp}. Therefore, by \eqref{eq:cuspcohdec}, $\Pi_f$ appears as an irreducible subspace of $H^q_{cusp,\omega}(G,E_\mu)$ for some degree $q$ and $\omega$ being the central character of $\Pi$ (where we can assume without loss of generality that $\omega|_{\R_+}=1$). Hence, by Prop.\ \ref{prop:inclusions}, $\Pi_f$ is even an irreducible submodule of $H^q_{cusp}(G,E_\mu)$. In view of the decomposition of the cohomology of $S_{G}$, cf. Cor.\ \ref{cor:cohdecP'phi}, and the essentially temperedness of the archimedean component of a cohomological, cuspidal representation $\Pi$, which by Strong Multiplicity One pins down $\Xi_\infty$ uniquely, one may hence reformulate Thm.\ \ref{thm:3.13} as follows:

{\it For any highest weight representation $E_\mu$ and any automorphism $\sigma\in \textrm{\emph{Aut}}(\C)$, the summand $H^q_{\{G\}}(G,E_\mu)$ of $H^q(S_{G},\mathcal E_\mu)$, being cuspidal cohomology, is mapped by $\sigma^*$ (cf. \ref{eq:sigmaiso}) isomorphically onto the summand
$H^q_{\{G\}}(G,{}^\sigma\!E_\mu)$ of $H^q(S_{G},{}^\sigma\!\mathcal E_\mu)$}

It is now clear why the following theorem of Franke is a generalization of Thm.\ \ref{thm:3.13}:

\begin{thm}[\cite{franke}, Thm.\ 20]\label{thm:3.13franke}
Let $E_\mu$ be a highest weight representation of $G_\infty$ and $\sigma\in \textrm{\emph{Aut}}(\C)$ an automorphism. For each associate class of parabolic $F$-subgroups $\{P\}$ the summand $H^q_{\{P\}}(G,E_\mu)$ of $H^q(S_{G},\mathcal E_\mu)$, is mapped by $\sigma^*$ isomorphically onto the corresponding summand $H^q_{\{P\}}(G,{}^\sigma\!E_\mu)$ of $H^q(S_{G},{}^\sigma\!\mathcal E_\mu)$. In other words, $\sigma^*$ respects the decomposition of sheaf cohomology along associate classes of parabolic $F$-subgroups.
\end{thm}

This latter theorem was refined further by Franke--Schwermer in \cite{schwfr}:

\begin{thm}[\cite{schwfr}, Thm.\ 4.3]\label{thm:3.13frankeschw}
Let $E_\mu$ be a highest weight representation of $G_\infty$ and $\sigma\in \textrm{\emph{Aut}}(\C)$ an automorphism. For each associate class of parabolic $F$-subgroups $\{P\}$, and each associate class of cuspidal automorphic representations $\varphi_{P}$, the summand $H^q_{\{P\},\varphi_{P}}(G,E_\mu)$ of $H^q(S_{G},\mathcal E_\mu)$, is mapped by $\sigma^*$ isomorphically onto the summand $H^q_{\{P\},{}^\sigma\!\varphi_{P}}(G,{}^\sigma\!E_\mu)$ of $H^q(S_{G},{}^\sigma\!\mathcal E_\mu)$ for a unique associate class ${}^\sigma\!\varphi_{P}$.
\end{thm}

We devote the next two subsections to showing certain generalizations of the above theorems to $G'$.

\subsection{A generalization of Theorems\ \ref{thm:3.13}, \ref{thm:3.13franke} and \ref{thm:3.13frankeschw}}\label{sect:cusptwist}

\begin{prop}\label{prop:cusptwist}
Let $\Pi'\in\mathscr D(G')$ be cuspidal and cohomological with respect to a highest weight representation $E_\mu$. Then for all $\sigma\in \textrm{\emph{Aut}}(\C)$, ${}^\sigma\Pi'_f$ is the finite part of a uniquely determined discrete series representation $\Xi'\in\mathscr D(G')$, which is cohomological with respect to ${}^\sigma\! E_\mu$. If $E_\mu$ is regular, then $\Xi'$ is cuspidal.
\end{prop}
\begin{proof}
By assumption $\Pi'_\infty$ is cohomological with respect to the highest weight representation $E_\mu$. As in Section \ref{sect:autcoh1}, by working with $G'(\A)^{(1)}$ instead of $G'(\A)$, we may henceforth assume without loss of generality that the central characters of $\Pi'_\infty$ and $E_\mu$ are both trivial on $\R_+$.

To proceed, let $b$ be a degree in which the $(\g_\infty',K'^\circ_\infty)$-cohomology of $\Pi'_\infty\otimes E_\mu$ does not vanish. Then $H^b(\g_\infty',K'^\circ_\infty,\Pi'_\infty\otimes E_\mu)$ is a non-trivial, but finite-dimensional $\C$-vector space on which $G'(\A_f)$ acts trivially. We may therefore embed
\begin{displaymath}
\xymatrix{
\Pi'_f \ar@{^{(}->}[r]  &  H^b(\g_\infty',K'^\circ_\infty,\Pi'_\infty\otimes E_\mu)\otimes\Pi'_f.}
\end{displaymath}
According to \eqref{eq:cuspcohdec}, this implies that $\Pi'_f$ appears as an irreducible $G'(\A_f)$-submodule in $H^b_{cusp,\omega}(G',E_\mu)$, with $\omega$ the central character of $\Pi'$. So, by Prop.\ \ref{prop:inclusions}, we may view $\Pi'_f$ as an irreducible $G'(\A_f)$-submodule of $H_!^b(S_{G'},\mathcal E_\mu)$. Recall from \eqref{eq:sigmaiso!} that for each $\sigma\in \textrm{Aut}(\C)$ there is a $\sigma$-linear, $G'(\A_f)$-equivariant isomorphism
$$\sigma^*_!: H_!^b(S_{G'},\mathcal E_\mu)\longrightarrow H_!^b(S_{G'},\mathcal {}^\sigma\! \mathcal E_\mu).$$
Hence, ${}^\sigma\Pi'_f$ appears as an irreducible submodule of $H_!^b(S_{G'},\mathcal {}^\sigma\! \mathcal E_\mu)$ and using Prop.\ \ref{prop:inclusions} again, even as an irreducible submodule of $\bar H^b_{dis}(G',{}^\sigma\! E_\mu)$. By \eqref{eq:disdecmp}, there must hence be a discrete series representation $\Xi'\in\mathscr D(G')$, such that $\Xi'_\infty$ is cohomological with respect to ${}^\sigma\! E_\mu$ and $\Xi'_f\cong {}^\sigma\Pi'_f$. By Strong Multiplicity One for discrete series automorphic representations of $G'(\A)$, cf.\ Badulescu--Renard \cite{ioanrenard} Th.\ 18.1.(b), this proves the first claim.

For the second claim, assume $E_\mu$ to be regular. Then ${}^{\sigma}\!E_\mu$ is regular, too. It is well--known that under these conditions each local archimedean component of $\Xi'_\infty$ must be essentially tempered. This follows from the Vogan--Zuckerman's condition \cite{vozu} (5.1), p.\ 73, (which is \eqref{eq:extending} here) together with the last paragraph on p.\ 58, {\it ibidem}; see also Li--Schwermer \cite{lischw}, Prop.\ 4.2, or Franke \cite{franke}, p.\ 258. Now, Wallach \cite{wallach}, Thm.\ 4.3, shows that $\Xi'$ is cuspidal.
\end{proof}

\begin{rem}
As it is obvious from the proof of Prop.\ \ref{prop:cusptwist}, $\Xi'$ will also be cuspidal under the weaker condition that $E_\mu$ only admits essentially tempered irreducible admissible representations of $G'_\infty$, which have non--vanishing $(\g'_\infty,K'^\circ_\infty)$-cohomology twisted by $E_\mu$. For an example of a non--regular coefficient module $E_\mu$, which has this property, see Grobner \cite{grob}, Prop.\ 3.5.
\end{rem}

\begin{thm}\label{thm:3.13'frankeschw}
Let $E_\mu$ be a regular highest weight representation of $G'_\infty$ and let $\sigma\in \textrm{\emph{Aut}}(\C).$ For each associate class of parabolic $F$-subgroups $\{P'\}$, and each associate class of cuspidal automorphic representations $\varphi_{P'}$, the summand $H^q_{\{P'\},\varphi_{P'}}(G',E_\mu)$ of $H^q(S_{G'},\mathcal E_\mu)$, is mapped by $\sigma^*$ isomorphically onto the summand $H^q_{\{P'\},{}^\sigma\!\varphi_{P'}}(G',{}^\sigma\!E_\mu)$ of $H^q(S_{G'},{}^\sigma\!\mathcal E_\mu)$ for a unique associate class ${}^\sigma\!\varphi_{P'}$:
\begin{displaymath}
\xymatrix{
H^q_{\{P'\},\varphi_{P'}}(G',E_\mu)\ar[r]^\cong_{\sigma*} & H^q_{\{P'\},{}^\sigma\!\varphi_{P'}}(G',{}^\sigma\!E_\mu).                   }
\end{displaymath}
Assuming $H^q_{\{P'\},\varphi_{P'}}(G',E_\mu)\neq 0$ and letting $\{P'\}$ be represented by the parabolic $F$-subgroup $P'$ with Levi factor $L'$ and $\varphi_{P'}$ be represented by the cuspidal automorphic representation $\Pi'$ of $L'(\A)$, then $\Pi'\otimes\rho_{P'}$ is cohomological. The $\sigma$-twist of its finite part $\Pi'_f\otimes\rho_{P'_f}$ is the finite part of a representation $\Xi'$ of $L'(\A)$, which is indeed cuspidal and the associate class ${}^\sigma\!\varphi_{P'}$ is uniquely determined by being represented by the representation $\Xi'\otimes\rho^{-1}_{P'}$.
\end{thm}
\begin{proof}
Recall that $G'=GL'_m$ for some $m\geq 1$. Let $\{P'\}$ be the associate class of the parabolic $F$-subgroup $P'$ with Levi factor $L'=\prod_{i=1}^r GL'_{m_i}$, $\sum_{i=1}^r m_i=m$ and $\varphi_{P'}$ be represented by the cuspidal automorphic representation $\Pi'$ of $L'(\A)$. Assume that $H^q_{\{P'\},\varphi_{P'}}(G',E_\mu)\neq 0$.

If $r=1$, i.e., $P'=G'$, each associate class $\varphi_{G'}$ of cuspidal automorphic representations is a singleton consisting of an equivalence class of a cuspidal automorphic representation $\Pi'$ of $G'(\A)$. Hence, if $H^q_{\{G'\},\varphi_{G'}}(G',E_\mu)\neq 0$, then all its irreducible $G'(\A_f)$-subquotients are in fact irreducible subrepresentations isomorphic to the finite part of the cohomological, cuspidal automorphic representation $\Pi'$, cf. Section \ref{sect:autcoh2}. So $\sigma^*$ maps such a $G'(\A_f)$-subquotient onto the $G'(\A_f)$-representation ${}^\sigma\Pi'_f$. According to Prop.\ \ref{prop:cusptwist}, ${}^\sigma\Pi'_f$ is the finite part of a discrete series representation $\Xi'$ which has non-zero cohomology with respect to ${}^\sigma\! E_\mu$. As $E_\mu$ is regular, $\Xi'$ is cuspidal by Prop.\ \ref{prop:cusptwist} and the assertion of the theorem follows in this case.

Now, let $r\neq1$, i.e., $P'\neq G'$. By the very construction of the space $\mathcal A_{\mathcal J,\{P'\},\varphi_{P'}}(G')$, cf. Section \ref{sect:automf}, and the regularity of $E_\mu$, there is an isomorphism of $G'(\A_f)$-modules
\begin{equation}\label{eq:erster}
H^q(\g'_\infty,K'^\circ_\infty,\textrm{Ind}_{P'(\A)}^{G'(\A)}[\tilde\Pi'\otimes S(\check\a^{G'}_{P',\C})]\otimes E_\mu) \ira H^q_{\{P'\},\varphi_{P'}}(G',E_\mu),
\end{equation}
where $\tilde\Pi'$ is as in Section \ref{sect:automf} and $S(\check\a^{G'}_{P',\C})$ is the algebra of differential operators on the finite-dimensional complex space $\check\a^{G'}_{P',\C}$ of variables of the Eisenstein series attached to $\tilde\Pi'$. All relevant details concerning this construction are contained in Franke--Schwermer \cite{schwfr} 1.2-1.4, 3.3 to which we refer. The existence of the above isomorphism is proved in Franke \cite{franke} Thm. 19 II. (See \cite{grob2} Cor.\ 16 for a version, which takes the given cuspidal support into account.) Let $\m'_\infty$ be the reductive Lie algebra obtained from $\l'_\infty$ by dividing out the diagonally embedded commutative Lie algebra $\a^{G'}_{P'}\hookrightarrow\l'_\infty$ and $K'_{L',\infty}$ be the intersection $K'^\circ_\infty\cap L'_\infty$. Then, in \cite{franke}, pp. 256-257, it is shown that as a $G'(\A_f)$-module,
\begin{equation}\label{eq:indec1}
H^q(\g'_\infty,K'^\circ_\infty,\textrm{Ind}_{P'(\A)}^{G'(\A)}[\tilde\Pi'\otimes S(\check\a^{G'}_{P',\C})]\otimes E_\mu)\cong
\end{equation}
$$ \textrm{Ind}_{P'(\A_f)}^{G'(\A_f)}\left[H^{q-l}(\m'_\infty,K'_{L',\infty},\tilde\Pi'_\infty\otimes E_{\mu_w,\m'_\infty})\otimes\Pi'_f\right],$$
where $E_{\mu_w,\m'_\infty}$ is a uniquely determined highest weight representation of $\m'_\infty$ and $l$ is a certain shift in degrees. In fact, $E_{\mu_w,\m'_\infty}$ is the restriction of the irreducible $L'_\infty$-representation $E_{\mu_w}$ of highest weight $\mu_w=w(\mu+\rho)-\rho$, $w$ being a uniquely determined Kostant representative in $W^{P'}$, cf. Borel--Wallach \cite{bowa} III 1.4 and III 3.3 for a definition of $W^{P'}$ and the uniqueness of $w$. Then, the shift in degrees $l$ equals the length $l(w)$ of $w$, see again \cite{bowa} III 3.3. On the other hand, directly by \cite{bowa}, III Thm.\ 3.3, one obtains that
\begin{equation}\label{eq:indec2}
H^q(\g'_\infty,K'^\circ_\infty,\textrm{Ind}_{P'(\A)}^{G'(\A)}[\Pi']\otimes E_\mu)\cong
\end{equation}
$$ \bigoplus_{a+b=q-l}\textrm{Ind}_{P'(\A_f)}^{G'(\A_f)}\left[H^{a}(\m'_\infty,K'_{L',\infty},\tilde\Pi'_\infty\otimes E_{\mu_w,\m'_\infty})\otimes\bigwedge^b\check{\a}^{G'}_{P',\C}\otimes\Pi'_f\right],$$
revealing \eqref{eq:indec1} as a direct summand in \eqref{eq:indec2} attached to $b=0$. As a consequence, there is also a surjective $G'(\A_f)$-module homomorphism
\begin{equation}\label{eq:indec3}
H^q(\g'_\infty,K'^\circ_\infty,\textrm{Ind}_{P'(\A)}^{G'(\A)}[\Pi']\otimes E_\mu) \twoheadrightarrow H^q_{\{P'\},\varphi_{P'}}(G',E_\mu).
\end{equation}
Now, invoking \cite{bowa} III 3.4.(5) and III 3.4.(14), we see that $H^{q-l}(\l'_\infty,K'_{L',\infty},\Pi'_\infty\otimes\rho_{P'_\infty}\otimes E_{\mu_w})\neq 0$, where $\rho_{P'_\infty}$ is the archimedean component of the character $\rho_{P'}=\rho_{P'_\infty}\otimes\rho_{P'_f}$, which is defined as usual via the Harish-Chandra homomorphism. Observe that the group $K'_{L',\infty}$ is in general neither connected nor contains the full connected component $Z'^\circ_{L',\infty}$ of the center of $L'_\infty$. However, a short moment of thought, using \cite{bowa} I 5.1 and Thm.\ I 5.3, shows that $\Pi_\infty'\otimes\rho_{P'_\infty}$ is really cohomological in our sense, made precise in Sect. \ref{sect:gKcoh}, with respect to the irreducible, algebraic representation $E_{\mu_w}$ of $L'_\infty$. One may check that $E_\mu$ being regular implies that also $E_{\mu_w}$ is regular. Therefore, we may apply Prop.\ \ref{prop:cusptwist} to the cuspidal automorphic representation $\Pi'\otimes\rho_{P'}$ and obtain a cuspidal automorphic representation $\Xi'$ of $L'(\A)$ with the property that $\Xi'_f\cong {}^\sigma\Pi'_f\otimes {}^\sigma\!\rho_{P'_f}$. The regularity of $E_{\mu_w}$ implies that ${}^\sigma\!E_{\mu_w}$ is regular, too, and hence the archimedean components of $\Xi'$ and $\Pi'\otimes\rho_{P'}$ are necessarily essentially tempered. Therefore, our Rem.\ \ref{rem:independent} shows that
$$H^b(\l'_\infty,(K'_{L',\infty}Z'_{L',\infty})^\circ,\Pi'_\infty\otimes\rho_{P'_\infty}\otimes E_{\mu_w})\cong H^b(\l'_\infty,(K'_{L',\infty}Z'_{L',\infty})^\circ,\Xi'_\infty\otimes {}^\sigma\! E_{\mu_w}),$$
for all degrees $b$. From this one may also derive that
\begin{equation}\label{eq:indec4}
H^b(\l'_\infty,(K'_{L',\infty})^\circ,\Pi'_\infty\otimes\rho_{P'_\infty}\otimes E_{\mu_w})\cong H^b(\l'_\infty,(K'_{L',\infty})^\circ,\Xi'_\infty\otimes {}^\sigma\! E_{\mu_w})
\end{equation}
holds for all degrees $b$. Keeping this in mind, we let $\partial_{P'}S_{G'}:=P'(F)\backslash G'(\A)/K'^\circ_\infty$. It is known that
$$H^q(\partial_{P'}S_{G'},\mathcal E_\mu)\cong {}^a\textrm{Ind}^{\pi_0(G'_\infty)\times G'(\A_f)}_{\pi_0(P'_\infty)\times P'(\A_f)}
\left [\bigoplus_{w\in W^{P'}} H^{q-l(w)}(L'(F)\backslash L'(\A)/K'_{L',\infty},\mathcal{E}_{\mu_w})\right ],$$
``${}^a\textrm{Ind}$'' denoting un-normalized or algebraic induction, cf. \cite{schw1447}, 7.1--7.2. This isomorphism allows one to define the $G'(\A_f)$- submodule $$H^q_{cusp}(\partial_{P'}S_{G'},\mathcal E_\mu):={}^a\textrm{Ind}^{\pi_0(G'_\infty)\times G'(\A_f)}_{\pi_0(P'_\infty)\times P'(\A_f)}
\left [\bigoplus_{w\in W^{P'}}\bigoplus_{\Phi'} H^{q-l(w)}(\l'_\infty, K'_{L',\infty}, \Phi'_\infty\otimes E_{\mu_w})\otimes\Phi'_f\right ],$$
the second sum ranging over all equivalence classes of cuspidal automorphic representations $\Phi'$ of $L'(\A)$ which satisfy $\omega_{\Phi'_\infty}|_{Z'^\circ_{L',\infty}}=\omega^{-1}_{E_{\mu_w}}|_{Z'^\circ_{L',\infty}}$. It is an easy observation that the kernel of the surjective map $\pi_0(P'_\infty)\ra\pi_0(G'_\infty)$ is equal to the image of $K'_{L',\infty}/(K'_{L',\infty})^\circ$ in $\pi_0(P'_\infty)$. Hence, one may rewrite the above by
$$\textrm{Ind}^{G'(\A_f)}_{P'(\A_f)}\left [\bigoplus_{w\in W^{P'}}\bigoplus_{\Phi'} H^{q-l(w)}(\l'_\infty, K'_{L',\infty}, \Phi'_\infty\otimes\rho_{P'_\infty}\otimes E_{\mu_w})\otimes\Phi'_f\right ]$$
$$\cong\bigoplus_{\Phi'} H^{q}(\g'_\infty,K'^\circ_\infty,\textrm{Ind}_{P'(\A)}^{G'(\A)}[\Phi']\otimes E_\mu),$$
where the $\Phi'$ now range over all equivalence classes of cuspidal automorphic representations of $L'(\A)$ which satisfy the modified condition $\omega_{\Phi'_\infty}|_{Z'^\circ_{L',\infty}}=\left(\omega_{E_{\mu_w}}\cdot\rho_{P'_\infty}\right)^{-1}|_{Z'^\circ_{L',\infty}}$ for the corresponding, uniquely determined Kostant representative $w\in W^{P'}$. In particular, by what we said above, $H^q(\g'_\infty,K'^\circ_\infty,\textrm{Ind}_{P'(\A)}^{G'(\A)}[\Pi']\otimes E_\mu)$ is a $G'(\A_f)$-submodule of $H^q_{cusp}(\partial_{P'}S_{G'},\mathcal E_\mu)$ and hence of the sheaf cohomology $H^q(\partial_{P'}S_{G'},\mathcal E_\mu)$. Whence, it makes sense to apply $\sigma^*$ to the latter $(\g'_\infty,K'^\circ_\infty)$-cohomology space. Choosing ${}^\sigma\!\varphi_{P'}$ to be the associate class which is represented by $\Xi'\otimes\rho_{P'}^{-1}$, we obtain the following diagram
\small
\begin{displaymath}
\xymatrix{
H^q(S_{G'},\mathcal E_\mu) \ar[d]^{\sigma^*} \ar@{<-^{)}}[r] & H^q_{\{P'\},\varphi_{P'}}(G',E_\mu) \ar@{<<-}[r]&  H^q(\g'_\infty,K'^\circ_\infty,\textrm{Ind}_{P'(\A)}^{G'(\A)}[\Pi']\otimes E_\mu) \ar@{^{(}->}[r] & H^q(\partial_{P'}S_{G'},\mathcal E_\mu)\ar[d]^{\sigma^*}\\
H^q(S_{G'},{}^\sigma\! \mathcal{E}_\mu) \ar@{<-^{)}}[r] & H^q_{\{P'\},{}^\sigma\!\varphi_{P'}}(G',{}^\sigma\! E_\mu) \ar@{<<-}[r]&  H^q(\g'_\infty,K'^\circ_\infty,\textrm{Ind}_{P'(\A)}^{G'(\A)}[\Xi'\otimes\rho_{P'}^{-1}]\otimes {}^\sigma\! E_\mu) \ar@{^{(}->}[r] & H^q(\partial_{P'}S_{G'}, {}^\sigma\! \mathcal E_\mu)              }
\end{displaymath}
\normalsize
Equation\ \eqref{eq:indec4} together with Strong Multiplicity One and Multiplicity One for discrete series representations of $L'(\A)$ finally show that the restriction of $\sigma^*:H^q(\partial_{P'}S_{G'},\mathcal E_\mu)\ra H^q(\partial_{P'}S_{G'},{}^\sigma\!\mathcal E_\mu)$ to the space $H^q(\g'_\infty,K'^\circ_\infty,\textrm{Ind}_{P'(\A)}^{G'(\A)}[\Pi']\otimes E_\mu)$ provides us with a commutative diagram
\begin{displaymath}
\xymatrix{
H^q(S_{G'},\mathcal E_\mu) \ar[d]^{\sigma^*} \ar@{<-^{)}}[r] & H^q_{\{P'\},\varphi_{P'}}(G',E_\mu) \ar@{<<-}[r]&  H^q(\g'_\infty,K'^\circ_\infty,\textrm{Ind}_{P'(\A)}^{G'(\A)}[\Pi']\otimes E_\mu) \ar[d]^{\sigma^*}\\
H^q(S_{G'},{}^\sigma\! \mathcal{E}_\mu) \ar@{<-^{)}}[r] & H^q_{\{P'\},{}^\sigma\!\varphi_{P'}}(G',{}^\sigma\! E_\mu) \ar@{<<-}[r]&  H^q(\g'_\infty,K'^\circ_\infty,\textrm{Ind}_{P'(\A)}^{G'(\A)}[\Xi'\otimes\rho_{P'}^{-1}]\otimes {}^\sigma\! E_\mu) }
\end{displaymath}
This settles the case $r\neq 1$ and hence the theorem.
\end{proof}

\begin{rem}
The regularity of $E_\mu$ is not only assumed for convenience. It guarantees the existence of the isomorphism \eqref{eq:erster}. This may fail without the regularity assumption on $E_\mu$. Indeed, for a general theorem showing this, one may refer the reader to Grobner \cite{grob2}, Thm.\ 22. Moreover, we would like to point out that -- in contrast to the split case, considered by Clozel -- if $E_\mu$ was not assumed to be regular, then $\Xi'_\infty$ (as well as $\Pi'_\infty$) did not need to be essentially tempered.
\end{rem}

\subsection{A complementary view on Thm.\ \ref{thm:3.13}}
Let now $\Pi'\in\mathscr D(G')$ be regular algebraic and assume that its global Jacquet-Langlands transfer $JL(\Pi')$ is cuspidal -- the standard assumption we had to make in Section \ref{sect:regalg}, in order to obtain proper generalizations of the various theorems on regular algebraic representations in the split case. The next theorem is complementary to Thm.\ \ref{thm:3.13'frankeschw} and shows that for any automorphism $\sigma\in \textrm{Aut}(\C)$, the $G'(\A_f)$-equivariant isomorphism $\sigma^*$ respects these properties. Hence, it may also be seen as another generalization of Thm.\ \ref{thm:3.13}.

\begin{thm}\label{thm:3.13'}
Let $\Pi'\in\mathscr D(G')$ be regular algebraic and assume that $JL(\Pi')$ is cuspidal. For all $\sigma\in \textrm{\emph{Aut}}(\C)$, there is a uniquely determined representation $\Xi'\in\mathscr D(G')$ as in Prop.\ \ref{prop:cusptwist}, which is in addition regular algebraic and $JL(\Xi')$ is cuspidal. The action of $\textrm{\emph{Aut}}(\C)$ commutes with taking the global Jacquet-Langlands transfer, i.e., ${}^\sigma\!JL(\Pi')=JL({}^\sigma\Pi')$ for all $\sigma\in \textrm{\emph{Aut}}(\C)$.
\end{thm}
\begin{proof}
Let $\Pi'=\Pi'_\infty\otimes\Pi'_f\in\mathscr D(G')$ be as in the statement of the theorem. By Thm.\ \ref{thm:regalgG'} we know that $\Pi'_\infty$ is cohomological and by
Cor.\ \ref{cor:globalJL} that $\Pi'$ is cuspidal. Hence, by Prop.\ \ref{prop:cusptwist}, there is a discrete series representation $\Xi'$ of $G'(\A)$, which satisfies $\Xi'_f\cong {}^\sigma\Pi'_f$ and has non--vanishing $(\g_\infty',K'^\circ_\infty)$-cohomology with respect to ${}^\sigma\! E_\mu$. Recall from Thm.\ \ref{thm:3.13} that there is a uniquely determined, regular algebraic, cuspidal automorphic representation $\Xi$ of $G(\A)$, which satisfies $\Xi_f\cong {}^\sigma\! (JL(\Pi')_f)$. Now we observe that at all split places $v\in V_f$
$$JL(\Xi')_v\cong \Xi'_v\cong ({}^\sigma\Pi'_f)_v \cong {}^\sigma\!(\Pi'_v)\cong {}^\sigma\!( JL(\Pi')_v)\cong ({}^\sigma\! JL(\Pi')_f)_v\cong \Xi_v.$$
Therefore, the discrete series representation $JL(\Xi')$ and the cuspidal representation $\Xi$ are isomorphic almost everywhere, and hence by the Strong Multiplicity One Theorem and the Multiplicity One Theorem for discrete series representations of $G(\A)$,
$$JL(\Xi')=\Xi.$$
In particular, $JL(\Xi')$ is cuspidal and regular algebraic, whence so is $\Xi'$. Since $\Xi={}^\sigma\!JL(\Pi')$ and $\Xi'={}^\sigma\Pi'$, the result follows.
\end{proof}

\begin{rem}
Our Thm.\ \ref{thm:3.13'} is also a proper generalization of Waldspurger's Thm.\ II.3.2 in \cite{waldsp}.
\end{rem}

\section{Results on rationality fields}
\subsection{}
The next theorem generalizes a result which is well--known in the split case (cf., e.g., Shimura \cite{shimura}, Harder \cite{hardermodsym} p.80, Waldspurger \cite{waldsp}, Cor.\ I.8.3 and first line of p.\ 153, and Clozel \cite{clozel}).

\begin{thm}\label{thm:Q(pi_f)numberfield}
Let $\Pi'\in\mathscr D(G')$ be cuspidal and cohomological. Then $\Q(\Pi'_f)$ is a number field.
\end{thm}
\begin{proof}
Let $\Pi'$ be cuspidal and assume that $\Pi'_\infty$ is cohomological with respect to $E_\mu$. By Prop.\ \ref{prop:cusptwist}, for any automorphism $\sigma\in \textrm{Aut}(\C)$, ${}^\sigma\Pi'_f$ is the finite part of a discrete series automorphic representation $\Xi'$ which has non--vanishing cohomology with respect to ${}^\sigma\! E_\mu$. Using Prop.\ \ref{prop:inclusions} we hence obtain inclusions
$${}^\sigma\Pi'_f\hookrightarrow H_{!}^q(G',{}^\sigma\! E_\mu),$$
for all $\sigma\in \textrm{Aut}(\C)$. Take an open, compact subgroup $K'_f$ of $G'(\A_f)$ which is small enough such that $\Pi'^{K'_f}_f\neq 0$. Then for each $\sigma$ also $({}^\sigma\Pi'_f)^{K'_f}\neq 0$ and $H_{!}^q(G',{}^\sigma\! E_\mu)^{K'_f}$ is finite-dimensional. Therefore, $|\textrm{Aut}(\C)/\mathfrak S(\Pi'_f)|$ is finite. By \cite{hung}, V Lem.\ 2.9, $|\textrm{Aut}(\C)/\mathfrak S(\Pi'_f)|$ is an upper bound for $\dim_\Q(\Q(\Pi'_f))=[\Q(\Pi'_f):\Q]$. So, $\Q(\Pi'_f)$ is a number field.
\end{proof}

\subsection{}\label{sect:Qpivs}
Let us also discuss how the local rationality fields $\Q(\Pi'_v)$, $v\in V_f$, and $\Q(\Pi'_f)$ are connected to each other.

\begin{prop}\label{prop:fieldcomp}
Let $\Pi'\in\mathscr D(G')$ be cuspidal and cohomological. Then $\Q(\Pi'_f)$ is the compositum of the fields $\Q(\Pi'_v)$, $v\in V_f-S$, $S\subset V_f$ an arbitrary finite subset containing all non-archimedean places where $\Pi'_v$ ramifies.
\end{prop}

The proof of Prop.\ \ref{prop:fieldcomp} requires a few preparatory lemmas. Let $v\in V_f$ be a split, non-archimedean place. So $G'_v=G_v=GL_n(F_v)$ and let $\varpi_v$ be an uniformizer for the maximal ideal of the ring of integers $\mathcal O_v$ of $F_v$. Since $\Pi_v$ is an unramified, irreducible admissible representation of $G'_v$, it is well--known that $\Pi_v$ is the unique Langlands quotient $J(\chi_1,...,\chi_n)$ of the induced representation $\textrm{Ind}_{B_v}^{G_v}[\chi_1\otimes...\otimes\chi_n],$
for some unramified characters $\chi_j: F_v^*\rightarrow\C^*$ and $B_v$ being the Borel subgroup of $G_v$. We put
$$\alpha^{\Pi'_v}_j:=|\varpi_v|_v^\frac{n-1}{2}\chi_j(\varpi_v),$$
with $|\cdot|_v$ being the normalized absolute value on $F_v$. Furthermore, for any automorphism $\sigma\in \textrm{Aut}(\C)$ denote by
$$\epsilon_{\sigma,v}=\frac{|\textrm{det}'|_v^{1/2}}{\sigma(|\textrm{det}'|_v^{1/2})}.$$
Then $\epsilon_{\sigma,v}$ is a quadratic character, because by the very definition $|\cdot|_v$ takes values in the rational numbers, whence $|\det'|_v=\sigma(|\det'|_v)$ for any automorphism $\sigma\in \textrm{Aut}(\C)$.
As a last ingredient, recall the {\it elementary symmetric polynomials} $f_1,....,f_n$, cf., e.g., \cite{hung} Appendix to Chp. V:
$$f_j(x_1,...,x_n)=\sum_{1\leq i_1<...<i_j\leq n}x_{i_1}x_{i_2}...x_{i_j}.$$
The next lemma is a generalization of Waldspurger \cite{waldsp}, Lem.\ I.2.3.

\begin{lem}\label{lem:Q(pi_v) als Q(...)}
Let $v\in V_f$ be a split, non-archimedean place and $\Pi_v$ an unramified, irreducible admissible representation of $G'_v=G_v=GL_n(F_v)$.  Then
$$\Q(\Pi_v)=\Q\left(f_1(\alpha^{\Pi'_v}_1,...,\alpha^{\Pi'_v}_n),...,f_n(\alpha^{\Pi'_v}_1,...,\alpha^{\Pi'_v}_n)\right).$$
\end{lem}
\begin{proof}
\begin{eqnarray*}
\sigma\in\mathfrak S(\Pi_v) & \Leftrightarrow & {}^\sigma\Pi_v\cong\Pi_v\\
& \Leftrightarrow & {}^\sigma\!J(\chi_1,...,\chi_n)\cong J(\chi_1,...,\chi_n)\\
& \underset{\textrm{\cite{clozel}, Lem.\ 3.5.(ii)}}{\Leftrightarrow} & J({}^\sigma\chi_1\epsilon_{\sigma,v}^{n-1},...,{}^\sigma\chi_n\epsilon_{\sigma,v}^{n-1})\cong
J(\chi_1,...,\chi_n)\\
& \Leftrightarrow & \{{}^\sigma\chi_1\epsilon_{\sigma,v}^{n-1},...,{}^\sigma\chi_n\epsilon_{\sigma,v}^{n-1}\}=\{\chi_1,...,\chi_n\}\\
& \Leftrightarrow & \sigma\left(f_j(\alpha^{\Pi'_v}_1,...,\alpha^{\Pi'_v}_n)\right)=f_j(\alpha^{\Pi'_v}_1,...,\alpha^{\Pi'_v}_n)\quad\forall 1\leq j\leq n\\
& \Leftrightarrow & \sigma\in \textrm{Aut}\left(\C\Big/\Q\left(f_1(\alpha^{\Pi'_v}_1,...,\alpha^{\Pi'_v}_n),...,f_n(\alpha^{\Pi'_v}_1,...,\alpha^{\Pi'_v}_n)\right)\right).
\end{eqnarray*}
Taking fixed fields proves the lemma.
\end{proof}

\begin{lem}\label{lem:2}
Let $v\in V_f$ be a split, non-archimedean place and $\Pi_v$ an unramified, irreducible admissible representation of $G'_v=G_v=GL_n(F_v)$. Assume that $\Q(\Pi_v)$ is a number field. Then for any finite Galois extension $\F_v$ of $\Q$, containing $\Q(\Pi_v)$,
$$\textrm{\emph{Gal}}(\F_v/\Q(\Pi_v))=\{\sigma\in \textrm{\emph{Gal}}(\F_v/\Q)| {}^{\tilde\sigma}\Pi_v\cong \Pi_v \textrm{ for all lifts } \tilde\sigma\in \textrm{\emph{Aut}}(\C)\}.$$
Moreover, the restriction map $\textrm{\emph{Aut}}(\C)\rightarrow \textrm{\emph{Gal}}(\F_v/\Q)$, sending $\tau$ to $\tau|_{\F_v}$, maps $\mathfrak S(\Pi_v)$ onto $\textrm{\emph{Gal}}(\F_v/\Q(\Pi_v))$.
\end{lem}
\begin{proof}
By Lem.\ \ref{lem:Q(pi_v) als Q(...)},
$${}^{\tilde\sigma}\Pi_v\cong\Pi_v \Leftrightarrow \tilde\sigma\left(f_j(\alpha^{\Pi'_v}_1,...,\alpha^{\Pi'_v}_n)\right)=f_j(\alpha^{\Pi'_v}_1,...,\alpha^{\Pi'_v}_n)\quad\forall 1\leq j\leq n.$$
But since $f_j(\alpha^{\Pi'_v}_1,...,\alpha^{\Pi'_v}_n)\in\Q(\Pi_v)\subseteq\F_v$ for all $1\leq j\leq n$, we obtain
$$\textrm{Gal}(\F_v/\Q(\Pi_v))=\{\sigma\in \textrm{Gal}(\F_v/\Q)| {}^{\tilde\sigma}\Pi_v\cong \Pi_v \textrm{ for all lifts } \tilde\sigma\in \textrm{Aut}(\C)\}.$$
The last assertion is obvious.
\end{proof}

\begin{lem}\label{lem:3}
Let $\Pi'\in\mathscr D(G')$ be cuspidal and cohomological and let $S\subset V_f$ be an arbitrary finite set containing all non-archimedean places where $\Pi'_v$ ramifies. Denote by $\E$ the compositum of the fields $\Q(\Pi'_v)$, $v\in V_f-S$. Then for any finite Galois extension $\F$ of $\Q$, containing $\Q(\Pi'_f)$, the restriction map $r_\F:\textrm{\emph{Aut}}(\C)\rightarrow \textrm{\emph{Gal}}(\F/\Q)$, sending $\tau$ to $\tau|_{\F}$, maps $\mathfrak S(\Pi'_f)$ onto $\textrm{\emph{Gal}}(\F/\E)$.
\end{lem}
\begin{proof}
Recall from Thm.\ \ref{thm:Q(pi_f)numberfield} that under the present assumptions $\Q(\Pi'_f)$ is a number field. So it makes sense to talk about a finite Galois extension $\F$ of $\Q$, containing $\Q(\Pi'_f)$. We get
\begin{eqnarray*}
\textrm{Gal}(\F/\E) & = & \bigcap_{v\in V_f-S} \textrm{Gal}(\F/\Q(\Pi'_v))\\
& \underset{\textrm{Lem.\ \ref{lem:2}}}{=} & \bigcap_{v\in V_f-S}\{\sigma\in \textrm{Gal}(\F/\Q)| {}^{\tilde\sigma}\Pi'_v\cong \Pi'_v \textrm{ for all lifts } \tilde\sigma\in \textrm{Aut}(\C)\}\\
& = & \bigcap_{v\in V_f}\{\sigma\in \textrm{Gal}(\F/\Q)| {}^{\tilde\sigma}\Pi'_v\cong \Pi'_v \textrm{ for all lifts } \tilde\sigma\in \textrm{Aut}(\C)\},
\end{eqnarray*}
where the last line uses Prop.\ \ref{prop:cusptwist} and Strong Multiplicity One for discrete series representations of $G'(\A)$. Therefore, we get even more that
$$\textrm{Gal}(\F/\E) = \{\sigma\in \textrm{Gal}(\F/\Q)| {}^{\tilde\sigma}\Pi'_f\cong \Pi'_f \textrm{ for all lifts } \tilde\sigma\in \textrm{Aut}(\C)\}$$
which is obviously the image of $\mathfrak S(\Pi'_f)$ under $r_{\F}$.
\end{proof}

We may now give the

\begin{proof}[Proof of Prop.\ \ref{prop:fieldcomp}:]
Let $\Pi'$ be cuspidal and cohomological. Let $\E$ be the compositum of the fields $\Q(\Pi'_v)$, $v\in V_f-S$. By Thm.\ \ref{thm:Q(pi_f)numberfield}, $\Q(\Pi'_f)$ is a number field. Let $\F$ be the Galois closure of $\Q(\Pi'_f)$. This is a finite Galois extension of $\Q$. We get
\begin{eqnarray*}
\Q(\Pi'_f) & = & \{z\in \C|\quad \sigma(z)=z\quad\forall\sigma\in\mathfrak S(\Pi'_f)\}\\
&= & \{z\in \F|\quad \sigma(z)=z\quad\forall\sigma\in r_\F(\mathfrak S(\Pi'_f))\}\\
& \underset{\textrm{Lem.\ \ref{lem:3}}}{=} & \{z\in \F|\quad \sigma(z)=z\quad\forall\sigma\in \textrm{Gal}(\F/\E)\}\\
& = & \E.
\end{eqnarray*}
\end{proof}

\begin{thm}\label{prop:cuspratio}
Let $\Pi'\in\mathscr D(G')$ be cuspidal and cohomological. Let $E_\mu$ be a highest weight representation of $G'_\infty$ with respect to which $\Pi'_\infty$ is cohomological. Then $\Pi'_f$ is defined over the composite field $\Q(\Pi')=\Q(\mu)\Q(\Pi'_f)$. In particular, $\Pi'_f$ is defined over a number field.
\end{thm}
\begin{proof}
By assumption $\Pi'_\infty$ is cohomological with respect to the highest weight representation $E_\mu$. As in Section \ref{sect:autcoh1}, by working with $G'(\A)^{(1)}$ instead of $G'(\A)$, we may henceforth assume without loss of generality that the central characters of $\Pi'_\infty$ and $E_\mu$ are both trivial on $\R_+$.

To proceed, let $b$ be the minimal degree in which the $(\g_\infty',K'^\circ_\infty)$-cohomology of $\Pi'_\infty\otimes E_\mu$ does not vanish and let $r$ be the number of split real places. The group $K'_\infty/K'^\circ_\infty\cong (\Z/2\Z)^r$ acts on the cohomology $H^b(\g_\infty',K'^\circ_\infty,\Pi'_\infty\otimes E_\mu)$. As in \cite{raghuram-shahidi-imrn}, 3.3, or \cite{mahnk}, 3.1.2, we may pick an $\epsilon\in (K'_\infty/K'^\circ_\infty)^*$ such that its isotypic component $H^b(\g_\infty',K'^\circ_\infty,\Pi'_\infty\otimes E_\mu)[\epsilon]$ becomes one-dimensional. Therefore, by Strong Multiplicity One and Multiplicity One for discrete series representations of $G'(\A)$, we may view $\Pi'_f$ in a canonical way as an irreducible $G'(\A_f)$-submodule of $H_!^b(S_{G'},\mathcal E_\mu)[\epsilon]$ -- the $\epsilon$-isotypic component of interior cohomology in degree $b$. Let $S\subset V_f$ be a finite set of places, such that $\Pi'_f$ is unramified outside $S$ and put $K'_S:=\prod_{v\notin S} GL_n(\mathcal O_v)$ and $G'(\A_{S}):=\prod'_{v\notin S} GL_n(F_v)$. For any subfield $\F\subset\C$ let
$$\mathcal H(K'_S,\F):=C_c^\infty\left(G'(\A_{S})/\!\!/K'_S,\F\right)$$
be the $K'_S$-leveled Hecke algebra of $\F$-valued functions. Hence, $H_!^b(S_{G'},\mathcal E_\mu)[\epsilon]^{K'_S}$ is a $\mathcal H(K'_S,\C)$-module in which $\Pi'^{K'_S}_f$ appears in a canonical way as a submodule. In Lem.\ \ref{lem:Qstructures}, we have established a $\Q(\mu)$-structure $H_!^b(S_{G'},\mathcal E_{\mu, \Q(\mu)})$ on interior cohomology. Letting $\Q(\Pi')=\Q(\mu)\Q(\Pi'_f)$ be the compositum of $\Q(\mu)$ and the rationality field of $\Pi'_f$, we obtain a $\Q(\Pi')$-structure
$H_!^b(S_{G'},\mathcal E_{\mu, \Q(\Pi')})$ and hence also a $\Q(\Pi')$-structure $H_!^b(S_{G'},\mathcal E_{\mu, \Q(\Pi')})[\epsilon]$. We now claim that the
$\mathcal H(K'_S,\Q(\Pi'))$-module $H_!^b(S_{G'},\mathcal E_{\mu, \Q(\Pi')})[\epsilon]^{K'_S}$ splits as a direct sum
$$H_!^b(S_{G'},\mathcal E_{\mu, \Q(\Pi')})[\epsilon]^{K'_S}\cong W_{\Pi'_f}\oplus W',$$
where $W_{\Pi'_f}\subset \Pi'^{K'_S}_f$ and the natural map $W_{\Pi'_f}\otimes_{\Q(\Pi')}\C\ra\Pi'^{K'_S}_f$ is an isomorphism of $\mathcal H(K'_S,\C)$-modules. This may be seen as follows: The Hecke algebra $\mathcal H(K'_S,\C)$ decomposes as a restricted tensor product $$\mathcal H(K'_S,\C)=\textrm{$\bigotimes$}'_{v\notin S}\mathcal H(K'_v,\C),$$
$K'_v=GL_n(\mathcal O_v)$, of local Hecke algebras $\mathcal H(K'_v,\C):=C_c^\infty(GL_n(F_v)/\!\!/K'_v,\C)$, each of which is spanned by the Hecke operators $T_{v,j}$, $1\leq j\leq n$, represented by the diagonal matrix $T_{v,j} = {\rm diag}(\varpi_v,\dots,\varpi_v,1,\dots,1)$
having exactly $j$-times $\varpi_v$ on the main diagonal. On a normalized spherical vector of $\Pi'_v$, $v\notin S$, $T_{v,j}$ operates by multiplication by the scalar $f_j(\alpha^{\Pi'_v}_1,...,\alpha^{\Pi'_v}_n)$, cf. Sect.\ \ref{sect:Qpivs}. Therefore, the functional
$$\lambda_{\Pi'_f,S}:\mathcal H(K'_S,\C)=\textrm{$\bigotimes$}'_{v\notin S}\mathcal H(K'_v,\C)\ra\C$$
determined by the local conditions
$$\lambda_{\Pi'_f,S}(T_{v,j})=f_j(\alpha^{\Pi'_v}_1,...,\alpha^{\Pi'_v}_n)$$
defines an eigenfunction for the action of $\mathcal H(K'_S,\C)$ on $H_!^b(S_{G'},\mathcal E_\mu)[\epsilon]^{K'_S}$. Recalling from Lem.\ \ref{lem:Q(pi_v) als Q(...)} that $\Q(\Pi'_v)=\Q\left(f_1(\alpha^{\Pi'_v}_1,...,\alpha^{\Pi'_v}_n),...,f_n(\alpha^{\Pi'_v}_1,...,\alpha^{\Pi'_v}_n)\right)$ and from Prop.\ \ref{prop:fieldcomp} that $\Q(\Pi'_f)$ is the compositum of the fields $\Q(\Pi'_v)$, $v\notin S$, we obtain that
$$H_!^b(S_{G'},\mathcal E_{\mu,\Q(\Pi')})[\epsilon]^{K'_S}\cong W(\lambda_{\Pi'_f,S})\oplus W_0,$$
where $W(\lambda_{\Pi'_f,S})$ is the eigenspace of $\lambda_{\Pi'_f,S}$. But this implies that there exists a $\mathcal H(K'_S,\Q(\Pi'))$-module $W_{\Pi'_f}$, which is a direct summand in $W(\lambda_{\Pi'_f,S})$ such that $W_{\Pi'_f}\subset \Pi'^{K'_S}_f$ and such that the natural map $W_{\Pi'_f}\otimes_{\Q(\Pi')}\C\ra\Pi'^{K'_S}_f$ is an isomorphism, which shows the above claim.

Now, let $V_{\Pi'_f}$ be the $\Q(\Pi')$-span of the $G'(\A_f)$-orbit of $W_{\Pi'_f}$ inside the $G'(\A_f)$-module $H_!^b(S_{G'},\mathcal E_{\Q(\Pi')})[\epsilon]$. In other words, if $\mathcal H(G'(\A_f),\Q(\Pi')):=C_c^\infty(G'(\A_f),\Q(\Pi'))$, then
$$V_{\Pi'_f}=\mathcal H(G'(\A_f),\Q(\Pi'))\cdot W_{\Pi'_f}.$$
We can finish the proof of the first assertion of theorem, if we show that the natural map
$$V_{\Pi'_f}\otimes_{\Q(\Pi')}\C\ra\Pi'_f$$
is an isomorphism. Injectivity follows from the following consideration:
\begin{eqnarray*}
\Pi'_f & = & \otimes_{v\in S}\Pi'_v\otimes \left(\mathcal H(G'(\A_S),\C)\cdot\left(\otimes'_{v\notin S}\Pi'_v\right)^{K'_S}\right)\\
 & = & \mathcal H(G'(\A_f),\C)\cdot\left(\Pi'^{K'_S}_f\right)\\
 & \cong & \mathcal H(G'(\A_f),\C)\cdot\left(W_{\Pi'_f}\otimes_{\Q(\Pi')}\C\right)\\
 & \supseteq & \left(\mathcal H(G'(\A_f),\Q(\Pi'))\cdot W_{\Pi'_f}\right)\otimes_{\Q(\Pi')}\C\\
 & = & V_{\Pi'_f}\otimes_{\Q(\Pi')}\C;
\end{eqnarray*}
while surjectivity is a consequence of the irreducibility of $\Pi'_f$. Hence, $V_{\Pi'_f}\otimes_{\Q(\Pi')}\C\cong\Pi'_f$.

Recalling Lem.\ \ref{lem:Erational} and Thm.\ \ref{thm:Q(pi_f)numberfield} shows that $\Q(\Pi')$ is a number field. Hence, also the second assertion of the theorem follows and the proof is complete.
\end{proof}

\begin{cor}
Let $\Pi\in\mathscr D(G)$ be cuspidal and cohomological. Then $\Pi_f$ is defined over $\Q(\Pi_f)$.
\end{cor}
\begin{proof}
By Clozel \cite{clozel}, p.122, $E_\mu$ is defined as a representation of $G(F)$ over $\Q(E_\mu)$, whence so are the $G(\A_f)$-modules $H^q(S_{G},\mathcal E_\mu)$ and $H_!^q(S_{G},\mathcal E_\mu)$. Our proof of Thm.\ \ref{prop:cuspratio} hence shows that $\Pi_f$ is defined over $\Q(\Pi_f)\Q(E_\mu)$. It hence suffices to prove that $\Q(\Pi_f)\supseteq\Q(E_\mu)$: If $\sigma\in$ Aut$(\C)$ fixes $\Pi_f$, i.e., is in $\mathfrak S(\Pi_f)$, it must fix the full representation ${}^\sigma\Pi:=\Xi$, by Strong Multiplicity One. In particular, ${}^\sigma\Pi_\infty\cong\Pi_\infty$. Since ${}^\sigma\Pi_\infty$ is cohomological with respect to ${}^\sigma\! E_\mu$ and $\Pi_\infty$ is cohomological with respect to $E_\mu$, this implies that ${}^\sigma\Pi_\infty$ is also cohomological with respect to $E_\mu$, too. As a consequence, ${}^\sigma\! E_\mu$ and $ E_\mu$ must have the same infinitesimal character, whence, as both of these representations are irreducible, finite-dimensional algebraic, necessarily ${}^\sigma\! E_\mu\cong E_\mu$. This implies that $\sigma\in\mathfrak S(E_\mu)$, or, otherwise put, $\mathfrak S(\Pi_f)\subseteq\mathfrak S(E_\mu)$. Taking fixed fields shows $\Q(\Pi_f)\supseteq\Q(E_\mu)$.
\end{proof}

Finally, let us investigate the interplay between rationality fields and the global Jacquet-Langlands transfer. Doing so, we get the following Proposition.

\begin{prop}\label{prop:Q=QJL}
Let $\Pi'\in\mathscr D(G')$ be regular algebraic and assume that $JL(\Pi')$ is cuspidal. Then
$$\Q(\Pi'_f)=\Q(JL(\Pi')_f).$$
\end{prop}
\begin{proof}
Let $\Pi'$ be regular algebraic such that $JL(\Pi')$ is cuspidal. Let $S$ be the finite set of non-archimedean places $v\in V_f$ where $\Pi'_v$ ramifies. By Cor.\ \ref{cor:globalJL}, $\Q(\Pi'_v)=\Q(JL(\Pi')_v)$ for all $v\in V_f-S$. The proposition now follows from Prop.\ \ref{prop:fieldcomp}.
\end{proof}

\subsection{}
The results of this paper suggest the following
generalization of Clozel's \cite[Conjectures 3.7 and 3.8]{clozel}:

\begin{conj}
Let $\Pi'\in\mathscr D(G')$ be cuspidal. Then the following are equivalent:
\begin{enumerate}
\item[(i)] $\Pi'_f$ is defined over a number field.
\item[(ii)] $\Pi'$ is algebraic.
\end{enumerate}
\end{conj}

\end{document}